\documentclass[onefignum,onetabnum]{siamart190516}

\usepackage[a4paper]{geometry}
\usepackage{mathrsfs}
\usepackage{amssymb}   
\usepackage{amsmath,amsfonts}
\usepackage{comment}
\usepackage[normalem]{ulem} \usepackage{graphicx}

\makeatletter

\newsiamremark{remark}{Remark}

\renewcommand{\figurename}{Fig.}

\newcommand*{\bbE}{\mathbb{E}}

\newcommand*{\bbN}{\mathbb{N}}

\newcommand*{\bs}[1]{\boldsymbol{#1}}

\newcommand*{\bsx}{\boldsymbol{x}}\newcommand*{\bsy}{\boldsymbol{y}}

\newcommand*{\cN}{\mathcal{N}}

\DeclareMathOperator*{\argmin}{arg\,min}  \DeclareMathOperator{\spn}{span}
\DeclareMathOperator{\supp}{supp}

\usepackage{xcolor}
\definecolor{darkred}{RGB}{139,0,0}
\definecolor{darkgreen}{RGB}{30,130,80}
\definecolor{darkmagenta}{RGB}{139,0,139}
\definecolor{darkorange}{RGB}{180,60,0}
\definecolor{darkcyan}{RGB}{0,139,139}
\newcommand{\rv}[1]{{{#1}}}

\newcommand*{\ykcst}[1]{\relax\ifmmode\text{\textcolor{darkgreen}{\sout{\ensuremath{#1}}}}\else\textcolor{darkgreen}{\sout{#1}}\fi}

\usepackage[hang,flushmargin]{footmisc} 

\AtBeginDocument{
	
	\DeclareBibliographyDriver{unpublished}{\usebibmacro{bibindex}\usebibmacro{begentry}\usebibmacro{author}\setunit{\labelnamepunct}\newblock
		\usebibmacro{title}\newunit
		\printlist{language}\newunit\newblock
		\usebibmacro{byauthor}\newunit\newblock
		\printfield{howpublished}\newunit\newblock
		\newunit\newblock
		\usebibmacro{location+date}\newunit\newblock \usebibmacro{doi+eprint+url}\newunit\newblock
		\printfield{annotation}\nopunct
		\setunit{\bibpagerefpunct}\newblock
		\usebibmacro{pageref}\usebibmacro{finentry}}
	
	\urlstyle{same}
	
	\makeatletter
	\DeclareFieldFormat{eprint:arxiv}{{\tt arXiv\addcolon
			\ifhyperref
			{\href{http://arxiv.org/\abx@arxivpath/#1}{\nolinkurl{#1}\iffieldundef{eprintclass}
					{}
					{\mkbibbrackets{\thefield{eprintclass}}}}}
			{\nolinkurl{#1}
				\iffieldundef{eprintclass}
				{}
				{\mkbibbrackets{\thefield{eprintclass}}}}}\nopunct}
	\makeatother
	
	\DeclareBibliographyDriver{unpublished}{\usebibmacro{bibindex}\usebibmacro{begentry}\usebibmacro{author}\setunit{\labelnamepunct}\newblock
		\usebibmacro{title}\newunit
		\printlist{language}\newunit\newblock
		\usebibmacro{byauthor}\newunit\newblock
		\printfield{howpublished}\newunit\newblock
		\newunit\newblock
		\usebibmacro{location+date}\newunit\newblock \usebibmacro{doi+eprint+url}\newunit\newblock
		\printfield{annotation}\nopunct
		\setunit{\bibpagerefpunct}\newblock
		\usebibmacro{pageref}\usebibmacro{finentry}}
	
	\renewbibmacro{in:}{\ifentrytype{article}{}{\printtext{\bibstring{in}\intitlepunct}}}
	\DeclareFieldFormat[article]{volume}{\textbf{#1}\nopunct\space}
	\DeclareFieldFormat[article]{number}{({#1})}
	
	\DeclareFieldFormat[article,unpublished]{title}{#1} 
	\DeclareFieldFormat[article]{pages}{#1}
	
	\AtEveryBibitem{\ifentrytype{article}{
			\clearfield{month}\clearfield{day}\clearfield{number}}{}
		\ifentrytype{book}{
			\clearfield{month}\clearfield{day}\clearfield{url}\clearfield{pagetotal}\clearlist{location}\clearfield{volume}\clearfield{urldate}\clearfield{review}\clearfield{series}}{}
		\ifentrytype{collection}{
			\clearfield{url}\clearfield{urldate}\clearfield{review}\clearfield{series}}{}
		\ifentrytype{incollection}{
			\clearfield{url}\clearfield{urldate}\clearfield{review}\clearfield{series}}{}
	}
}

\makeatother

\usepackage[style=numeric,style=numeric,
sorting=nyt,maxbibnames=99,firstinits=true,url=false,doi=false,isbn=false]{biblatex}
\usepackage{xpatch}
\xpatchfieldformat{url}{\mkbibacro{URL}}{\space}{}{}

\author{
	{Yoshihito Kazashi}\footnotemark[2] and 
	{Fabio Nobile}\footnotemark[2]}

\headers{Density estimation in RKHS}{Y. Kazashi and F. Nobile}

\title{Density estimation in RKHS\\ with application to Korobov spaces in high dimensions}

\author{Yoshihito Kazashi \thanks{
		CSQI, Institute of Mathematics, \'Ecole Polytechnique F\'ed\'erale de Lausanne, 1015 Lausanne, Switzerland (\email{y.kazashi@uni-heidelberg.de}).}\\
	\and
	Fabio Nobile \thanks{
		CSQI, Institute of Mathematics, \'Ecole Polytechnique F\'ed\'erale de Lausanne, 1015 Lausanne, Switzerland (\email{fabio.nobile@epfl.ch}).}\\
}

\begin{document}
	
	\maketitle
	
	\begin{abstract}
		A kernel method for estimating a probability density function (pdf)
		from an i.i.d.~sample drawn from such density is presented. Our estimator
		is a linear combination of kernel functions, the coefficients of which
		are determined by a linear equation. An error analysis for the mean
		integrated squared error is established in a general reproducing kernel
		Hilbert space setting. The theory developed is then applied to estimate
		pdfs belonging to weighted Korobov spaces, for which a dimension independent 
		convergence rate is established. Under a suitable smoothness assumption,
		our method attains a rate arbitrarily close to the optimal rate. Numerical
		results support our theory.
	\end{abstract}
	
	
	
	\section{Introduction}
	In this paper, we propose and analyse a kernel-based method to approximate
	probability density functions on a domain of an arbitrary dimension.
	Density approximations have a long history \cite{Scott.D_2015_MultivariateDensityEstimation,Wand.M.P_Jones_1995_book_KernelSmoothing}.
	However, these classical methods typically suffer from the so-called
	curse of dimensionality, i.e.~their error convergence rates deteriorate
	in dimension, and thus their practical use is limited to relatively
	low-dimensional settings; see for example \cite{Scott.D_2015_MultivariateDensityEstimation,Wand.M.P_Jones_1995_book_KernelSmoothing}
	for more details. Over the past decade there has been increasing interest
	in studying random variables taking values in high-dimensional spaces.
	For instance, in Uncertainty Quantification applications one is typically
	interested to study statistics of the solution of a complex differential
	model that contains random components; see for example \cite{Babuska.I_Nobile_Tempone_2010_SIREV,Cohen.A_DeVore_2015_Acta,Kuo.F.Y_Nuyens_2016_FoCM_survey}.
	Our work is partly inspired by this type of applications.
	
	Let $(\Omega,\mathcal{F},\mathbb{P})$ be a probability space and
	$(D,\mathcal{B})$ be a measurable space. Given independent random
	variables $Y_{1},\dots,Y_{M}\colon\Omega\to D$ that follow an identical
	distribution defined by a density $f$ with respect to a measure $\mu$
	on $\mathcal{B}$, we aim to approximate $f$ with a positive definite
	kernel $K(\cdot,\cdot)\colon D\times D\to\mathbb{R}$. In particular,
	we seek for an approximation of the form
	\begin{equation}
		f(\cdot)\approx\sum_{k=1}^{N}c_{k}(\omega)K(x_{k},\cdot),\label{eq:approx-form}
	\end{equation}
	where $N$ is a positive integer, $X:=\{x_{k}\mid k=1,\dots,N\}$
	is a pre-selected point set in $D$, and the (random) coefficients
	$c_{1}(\omega),\dots,c_{N}(\omega)\in\mathbb{R}$, which depend on
	the sample $\boldsymbol{Y}(\omega):=(Y_{1}(\omega),\dots,Y_{M}(\omega))$,
	are determined by solving a linear equation. More precisely, we denote
	the approximate density of the form \eqref{eq:approx-form} by $f_{M,N;\boldsymbol{Y}}^{\lambda}$
	and construct it as the solution of the following problem: Find $f_{M,N;\boldsymbol{Y}}^{\lambda}\in V_{N}$
	such that 
	\begin{equation}
		\langle f_{M,N;\boldsymbol{Y}}^{\lambda},v\rangle_{L_{\mu}^{2}(D)}+\lambda\langle f_{M,N;\boldsymbol{Y}}^{\lambda},v\rangle_{K}=\frac{1}{M}\sum_{m=1}^{M}v(Y_{m})\quad\text{for all }v\in V_{N},\label{eq:eq-disc-Intro}
	\end{equation}
	where $\lambda>0$ is a ``regularization'' parameter, and \[V_{N}:=V_{N}(X):=\mathrm{span}\{K(x_{k},\cdot)\mid k=1,\dots,N\}.\]
	Here, $\langle\cdot,\cdot\rangle_{L_{\mu}^{2}(D)}$ is the $L^{2}$-inner
	product with respect to the measure $\mu$ on $(D,\mathcal{B})$,
	and $\langle\cdot,\cdot\rangle_{K}$ is the inner product of the reproducing
	kernel Hilbert space (RKHS) $\mathcal{N}_{K}$ defined by $K$. The
	set of points $X:=\{x_{k}\}_{k=1}^{N}\subset D$ determines the approximation
	space $V_{N}(X)$, and they should be chosen carefully. More details
	will be discussed in Section~\ref{sec:2}. 
	
	The approximation of type \eqref{eq:eq-disc-Intro} is a variant of
	what Hegland et al.~proposed in \cite{Hegland.M_Hooker_Roberts_2000_density},
	in which a standard finite element space was considered as the approximation
	space. As such, the method proposed in \cite{Hegland.M_Hooker_Roberts_2000_density}
	becomes infeasible when the dimension of $D$ is large. Peherstorfer
	et al.~\cite{Peherstorfer.B_Plfuege_Bungartz_2014_densityadaptiveSG}
	considered sparse-grid basis functions instead of the standard finite
	element basis functions in the method of \cite{Hegland.M_Hooker_Roberts_2000_density},
	to deal with larger dimensions, but the approximation error and its
	dependence on the dimension is not investigated. Roberts and Bolt
	\cite{Roberts.S.G_Bolt_2009_densitySparseGrids} considered a method
	in the same vein as~\cite{Peherstorfer.B_Plfuege_Bungartz_2014_densityadaptiveSG},
	but without the regularization term. They outline an error analysis,
	but their claimed estimates will result in a mean integrated squared
	error (MISE) decaying as $\mathcal{O}_{d}(|\log M|^{2d}M^{-4/5})$
	with a constant exponentially increasing in $d$. Another class of
	density estimators that have been developed e.g.~in \cite{Griebel.M_Hegland_2010_prior,Wong.M_Hegland_2013_MAP}
	is the MAP estimator. The methods in \cite{Griebel.M_Hegland_2010_prior,Wong.M_Hegland_2013_MAP}
	involve minimising a non-linear functional via Newton's method, whereas
	our method only involves solving the linear equation \eqref{eq:eq-disc-Intro}.
	Moreover, the method in~\cite{Griebel.M_Hegland_2010_prior} is limited
	to three dimension as the dimension of the domain of the target density,
	and in~\cite{Wong.M_Hegland_2013_MAP} the approximation error is
	not investigated. In contrast to these works, as we will see later
	in Section~\ref{sec:Example}, under suitable smoothness/periodicity
	assumptions on the density $f$ we will establish a faster, \emph{dimension-independent}
	error decay in terms of MISE.
	
	To analyse the error of our method, we first derive a general theory
	in a RKHS setting. Under the assumption that the target density function
	is in the RKHS associated with the kernel $K$, we will establish
	an MISE bound. It turns out that the bias can be bounded by an orthogonal
	projection error plus a regularization term, while the variance decays
	at a rate arbitrarily close to $M^{-1}$ provided that $f$ is sufficiently
	``smooth''. More precisely, we have
	\[
	(\text{MISE})\leq\|\mathscr{P}_{N}f-f\|_{L_{\mu}^{2}(D)}^{2}+\mathcal{O}(\lambda^{2}+M^{-1}\lambda^{-\tau}),
	\]
	where $\mathscr{P}_{N}$ is the $\mathcal{N}_{K}$-orthogonal projection
	onto $V_{N}$, and given that $f$ is in a sufficiently small RKHS,
	$\tau\in(0,1]$ can be taken arbitrarily small; see Theorem~\ref{thm:general-summary}.
	Since a large number of projection error estimates are readily available
	for various kernels, this estimate will directly provide estimates
	on the MISE. 
	For $f$ smooth, the projection error typically decays
	fast in $N$. Hence, in such cases, for the optimal choice of $N$
	and $\lambda$ depending on the sample size $M$, the MISE decays
	at a rate arbitrarily close to $M^{-1}$. 
	Such a rate may be called
	\emph{near-optimal}, since, in view of the lower bound in \cite{Boyd.D.W_Steele_1978_LowerBoundsNonparametricDensityEstimation},
	the rate $M^{-1}$ is optimal.
	
	We will demonstrate the strength of this theory through an example.
	We will apply the theory to the so-called Korobov kernel and the corresponding
	space, which is, roughly speaking, a Sobolev space with periodicity.
	Given that the target density is in this space, it turns out that
	the bias can be bounded by the interpolation error up to a regularization
	term, where the interpolation points are $\{x_{k}\}_{k=1}^{N}$. Note
	that the kernel interpolation is optimal among all approximations
	that use only the same function values of $f$, in the sense that
	it gives the least possible worst-case error in any norm that is no
	stronger than the RKHS-norm; see for example \cite[Theorem 2,][]{KaarniojaEtAl.V_2021_FastApproximationPeriodic}
	for a proof of this optimality result for the same setting as this
	paper. As such, the interpolation error can be bounded by the approximation
	error delivered by other algorithms that use evaluations at the same
	set of points $\{x_{k}\}_{k=1}^{N}$. Approximation errors in Korobov
	spaces by kernel interpolation have been extensively studied \cite{Zeng.X_Leung_Hickernell_2006_SplineLattice,Zeng.X_Kritzer_Hickernell_2009_SplineLatticeDigitalNets,KaarniojaEtAl.V_2021_FastApproximationPeriodic}.
	In this paper, following \cite{KaarniojaEtAl.V_2021_FastApproximationPeriodic},
	as $\{x_{k}\}_{k=1}^{N}$ we choose the so-called rank-$1$ lattice
	points. Moreover, like \cite{KaarniojaEtAl.V_2021_FastApproximationPeriodic},
	we consider target (density) functions with a favourable anisotropy
	structure by assuming that they are in a \textit{weighted} space.
	By exploiting this structure, we establish convergence-rate estimates
	that are \emph{independent of the dimension}. 
	\rv{When the smoothness parameter of the Korobov space is an even integer, the rate established turns out to be \emph{asymptotically minimax} up to an arbitrarily small $\epsilon>0$.}  
	Moreover, the lattice
	structure gives a circulant matrix for the linear equation, which
	makes solving the equation fast. Numerical results support our theory.
	
	Random variables having a periodic density function arise for example as circular observations. Although they are important in many applications such as biology, geology, and political science \cite{TopicsCircularStatBook,Agostinelli.C_2007_RobustEstimationCircular,GillEtAl.J_2010_CircularDataPolitical,NodehiEtAl.A_2021_EstimationParametersMultivariate}, estimating such density function in high dimensional setting remains a challenge \cite{DiMarzioEtAl.M_2011_KernelDensityEstimation}.
	Moreover, we note that, if the target density is compactly supported and smooth,
	then we can always normalize the sample domain and assume a periodic 
	extension. 
	We mention several other theoretical results, although for
	methods different from ours, on periodic density estimations;  see for
	example \cite[Chapter 12]{Devroye.L_Gyoenfi_1985_book_NonparametricDensity},
	\cite{Wahba.G_1981_DatabasedOptimalSmoothing}, and as a special case,
	compactly supported density functions \cite{Walter.G.G_1977_PropertiesHermite}.
	In particular, periodic Sobolev density functions have been considered
	in \cite{Wahba.G_1981_DatabasedOptimalSmoothing} for the one dimensional
	case, where the author suggests that in many applications it might
	be preferable to assume the true density has compact support and to
	scale the data to the interior of $[0,1]$. We note that the paper
	\cite{Wahba.G_1981_DatabasedOptimalSmoothing} briefly addresses the
	multi-dimensional case. Their results do not exploit the anisotropic
	structure of the target density function, and the MISE rate proved
	there severely suffers from the curse of dimensionality, unlike ours.
	
	In passing, we note that our approximation \eqref{eq:approx-form}
	does not, in general, give a non-negative density function nor does
	it integrate to $1$. We remind the reader that satisfying these conditions
	is already an issue in the standard kernel density estimation in one
	dimension; see for example \cite{Terrel.G.R_Scott_1980_ImprovingConvergence}
	for discussions on relaxing these conditions to obtain a MISE convergence
	rate faster than $O(M^{-4/5})$, where $M$ is the sample size.
	
	The rest of the paper is organized as follows. Section~\ref{sec:2}
	introduces the problem setting and our method. An error analysis in
	a RKHS setting is presented in Section~\ref{sec:gen-theory}. Then
	in Section~\ref{sec:Example} we will apply this theory to the Korobov
	space setting, and establish a dimension independent MISE decay rate.
	Numerical results in Section~\ref{sec:numerical} support our theory,
	and Section~\ref{sec:Conclusions} concludes the paper.
	
	\section{Density approximation using kernels\label{sec:2}}
	
	\subsection{Reproducing kernel Hilbert space}
	
	Let $(D,\mathcal{B},\mu)$ be a measure space, and let $(\mathcal{N}_{K},\langle\cdot,\cdot\rangle_{K},\|\cdot\|_{K})$
	denote the reproducing kernel Hilbert space (RKHS) associated with
	the positive definite kernel $K\colon D\times D\to\mathbb{R}$\rv{, 
		i.e.,  $K(x,x')=K(x',x)$ for all $x,x'\in D$, and for any $m\in\mathbb{N}$, $t_j,t_k\in\mathbb{R}$, and $x_j,x_k\in D$, $j,k=0,\dots,m$, we have 
		$\sum_{j,k=0}^m t_j K(x_j,x_k) t_k \geq  0$.
	}
	This
	kernel may possibly be unbounded, but we assume $\int_{D}\sqrt{K(\rv{x},\rv{x})}\mathrm{d}\mu(\rv{x})<\infty$
	and $\int_{D}K(\rv{x},\rv{x})\mathrm{d}\mu(\rv{x})<\infty$. The first condition
	ensures
	\begin{align*}
		\int_{D}|K(x,\rv{x'})|\mathrm{d}\mu(\rv{x'}) & =\int_{D}|\langle K(\cdot,x),K(\cdot,\rv{x'})\rangle_{K}|\mathrm{d}\mu(\rv{x'})\\
		& \leq\sqrt{K(x,x)}\int_{D}\sqrt{K(\rv{x'},\rv{x'})}\mathrm{d}\mu(\rv{x'})\quad\text{for any }x\in D
	\end{align*}
	so that we have $f_{M,N;\boldsymbol{Y}}^{\lambda}\in L_{\mu}^{1}(D)$,
	while the second ensures that every $g\in\mathcal{N}_{K}$ is $\mu$-square
	integrable:
	\[
	\int_{D}|g(\rv{x})|^{2}\mathrm{d}\mu(\rv{x})\leq\|g\|_{K}^{2}\int_{D}\|K(\cdot,\rv{x})\|_{K}^{2}\mathrm{d}\mu(\rv{x})=\|g\|_{K}^{2}\int_{D}K(\rv{x},\rv{x})\mathrm{d}\mu(\rv{x})<\infty.
	\]
	Moreover, throughout this paper we will assume that $K$ admits a
	representation
	\begin{equation}
		K(x,x')=\sum_{\ell=0}^{\infty}\beta_{\ell}\varphi_{\ell}(x)\varphi_{\ell}(x')\quad x,x'\in D,\label{eq:K-series}
	\end{equation}
	with a positive sequence $(\beta_{\ell})_{\ell=0}^{\infty}\subset(0,\infty)$
	converging to $0$, and a complete orthonormal system $\{\sqrt{\beta_{\ell}}\varphi_{\ell}\}$
	of $\mathcal{N}_{K}$ such that the series is absolutely (point-wise)
	convergent and that $\{\varphi_{\ell}\}$ is an orthonormal system
	of $L_{\mu}^{2}(D)$. Then, the inner product for $\mathcal{N}_{K}$
	may be represented by
	\[
	\langle f,g\rangle_{K}=\sum_{\ell=0}^{\infty}\frac{\langle f,\varphi_{\ell}\rangle_{L_{\mu}^{2}(D)}\,\langle g,\varphi_{\ell}\rangle_{L_{\mu}^{2}(D)}}{\beta_{\ell}},
	\]
	where we used the notation $\langle u,w\rangle_{L_{\mu}^{2}(D)}:=\int_{D}u(x)v(x)\mathrm{d}\mu(x)$
	for the $L_{\mu}^{2}(D)$-inner product.
	
	For example, suppose that $(D,\mathcal{B})$ is a Hausdorff topological
	space with the corresponding Borel $\sigma$-algebra, $\mu$ is strictly
	positive, i.e.~$\mu(O)>0$ for any nonempty open set $O\subset D$,
	and that $K\colon D\times D\to\mathbb{R}$ is continuous. Then, the
	kernel $K$ admits a representation \eqref{eq:K-series} with an absolutely
	convergent series. Indeed, the condition $\int_{D}K(x,x)\mathrm{d}\mu(x)<\infty$
	ensures that $\mathcal{N}_{K}$ is compactly embedded into $L_{\mu}^{2}(D)$
	\cite[Lemma 2.3]{SteinwartEtAl.I_2012_MercerTheoremGeneral}. In turn,
	\cite[Lemma 2.2]{SteinwartEtAl.I_2012_MercerTheoremGeneral} implies
	that the integral operator $T_{K}\colon L_{\mu}^{2}(D)\to L_{\mu}^{2}(D)$
	defined by
	\[
	T_{K}g:=\int_{D}K(\cdot,x)g(x)\,\mathrm{d}\mu(x),\quad g\in L_{\mu}^{2}(D)
	\]
	is compact, and thus we can use representatives of the corresponding
	eigensystem to construct the representation \eqref{eq:K-series} (see
	\cite[Lemma 2.12]{SteinwartEtAl.I_2012_MercerTheoremGeneral} and
	\cite[Corollary 3.5]{SteinwartEtAl.I_2012_MercerTheoremGeneral}).
	We defer to \cite{SteinwartEtAl.I_2012_MercerTheoremGeneral} for
	more general conditions that imply the representation \eqref{eq:K-series}.
	
	For later use, we introduce the notation
	\begin{equation}
		\langle u,w\rangle_{\lambda}:=\langle u,w\rangle_{L_{\mu}^{2}(D)}+\lambda\langle u,w\rangle_{K}\quad\text{for }u,w\in\mathcal{N}_{K},\label{eq:ip-lambda}
	\end{equation}
	where $\lambda>0$ is a parameter. The bilinear form $\langle\cdot,\cdot\rangle_{\lambda}$
	is an inner product on $\mathcal{N}_{K}$ and $\|\cdot\|_{\lambda}:=\sqrt{\rv{\langle}\cdot,\cdot\rv{\rangle}_{\lambda}}$
	is equivalent to $\|\cdot\|_{K}$: for $v,w\in\mathcal{N}_{K}$ we
	have $\langle v,v\rangle_{\lambda}\geq\lambda\|v\|_{K}$ and $|\langle v,w\rangle_{\lambda}|\leq(\sup_{\ell\geq0}\beta_{\ell}+\lambda)\|v\|_{K}\|w\|_{K}$.
	
	We also introduce a continuum scale of nested Hilbert spaces related
	to $\mathcal{N}_{K}$. For $\tau>0$ we denote by $\cN_{K}^{\tau}$
	the normed space $\mathcal{N}_{K}^{\tau}:=\{v\in L_{\mu}^{2}(D)\mid\,\|v\|_{\mathcal{N}_{K}^{\tau}}<\infty\}$
	with $\|v\|_{\mathcal{N}_{K}^{\tau}}:=\big(\sum_{\ell=0}^{\infty}\beta_{\ell}^{-\tau}|\langle v,\varphi_{\ell}\rangle_{L_{\mu}^{2}(D)}|^{2}\big)^{1/2}$,
	where $(\beta_{\ell})_{\ell=0}^{\infty}\subset(0,\infty)$ is as in~\eqref{eq:K-series}.
	Note that if $\tau>0$ is such that 
	\begin{equation}
		\sum_{\ell=0}^{\infty}\beta_{\ell}^{\tau}\varphi_{\ell}(x)^{2}<\infty\qquad\text{for all }\ x\in D\label{eq:cond-rep}
	\end{equation}
	then the series $\sum_{\ell=0}^{\infty}\langle v,\varphi_{\ell}\rangle_{L_{\mu}^{2}(D)}\varphi_{\ell}(x)$
	for $v\in\mathcal{N}_{K}^{\tau}$ is (point-wise) absolutely convergent
	for any $x\in D$. In this case, we understand $v\in\mathcal{N}_{K}^{\tau}\subset L_{\mu}^{2}(D)$
	as the representative of the corresponding equivalence class in $L_{\mu}^{2}(D)$
	specified by this series. Then, we have $\mathcal{N}_{K}^{1}=\mathcal{N}_{K}$.
	
	We denote by $(\mathcal{N}_{K}^{\tau})'$ the topological dual space
	of $\mathcal{N}_{K}^{\tau}$. Moreover, we consider the normed space
	\begin{align*}
		&\cN_{K}^{-\tau}\\
		&:=\left\{\Psi  \colon\mathcal{N}_{K}^{\tau}\to\mathbb{R},\Psi(v):=\sum_{\ell=0}^{\infty}\Psi_{\ell}\langle\varphi_{\ell},v\rangle_{L_{\mu}^{2}(D)}
		\,\left\vert \,\begin{array}{l}
			(\Psi_{\ell})_{\ell\geq0}\subset\mathbb{R}\text{ such that}\\
			\|\Psi\|_{\cN_{K}^{-\tau}}<\infty
		\end{array}\right\}, 
		\right.
	\end{align*}
	where $\|\Psi\|_{\cN_{K}^{-\tau}}:=\big(\sum_{\ell=0}^{\infty}\beta_{\ell}^{\tau}\Psi_{\ell}^{2}\big)^{1/2}$.
	We will use the following characterisation of $(\cN_{K}^{\tau})'$.
	\begin{proposition}
		\label{prop:dual-isom}For $\tau\in(0,1]$, the dual space $(\cN_{K}^{\tau})'$
		equipped with the functional norm is isometrically isomorphic to $\cN_{K}^{-\tau}$.
	\end{proposition}
	\begin{proof}
		First, we show that $\cN_{K}^{-\tau}$ is a vector subspace of $(\cN_{K}^{\tau})'$.
		Indeed, for $\Psi\in\cN_{K}^{-\tau}$ and $v\in\cN_{K}^{\tau}$ we
		have
		\begin{equation}
			|\Psi(v)|\leq\biggl(\sum_{\ell=0}^{\infty}\beta_{\ell}^{\tau}\Psi_{\ell}^{2}\biggr)^{1/2}\|v\|_{\mathcal{N}_{K}^{\tau}}=\|\Psi\|_{\mathcal{N}_{K}^{-\tau}}\|v\|_{\mathcal{N}_{K}^{\tau}}<\infty,\label{eq:issubspace}
		\end{equation}
		\sloppy{so that $\|\Psi\|_{(\cN_{K}^{\tau})'}\leq\|\Psi\|_{\mathcal{N}_{K}^{-\tau}}$
			and $\cN_{K}^{-\tau}\subset(\cN_{K}^{\tau})'$. Next, take $\Phi\in(\mathcal{N}_{K}^{\tau})'$
			and $v=\sum_{\ell=0}^{\infty}\langle v,\varphi_{\ell}\rangle_{L_{\mu}^{2}(D)}\varphi_{\ell}\in\mathcal{N}_{K}^{\tau}$
			arbitrarily, where we note that $v\in\mathcal{N}_{K}^{\tau}$ implies
			that this series is convergent in $\mathcal{N}_{K}^{\tau}$. Then,
			the continuity of $\Phi$ implies $\Phi(v)=\sum_{\ell=0}^{\infty}\Phi(\varphi_{\ell})\langle v,\varphi_{\ell}\rangle_{L_{\mu}^{2}(D)}\in\mathbb{R}$.
			To show $\Phi\in\mathcal{N}_{K}^{-\tau}$, we note that for any $L\in\mathbb{N}$
			we have}
		\begin{align*}
			0\leq\sum_{\ell=0}^{L}\beta_{\ell}^{\tau}|\Phi(\varphi_{\ell})|^{2} & =\sum_{\ell=0}^{L}\beta_{\ell}^{\tau}\Phi(\varphi_{\ell})\,\Phi(\varphi_{\ell})=\Phi\Bigl(\sum_{\ell=0}^{L}\beta_{\ell}^{\tau}\Phi(\varphi_{\ell})\varphi_{\ell}\Bigr)\\
			& \leq\|\Phi\|_{(\mathcal{N}_{K}^{\tau})'}\biggl\|\sum_{\ell=0}^{L}\beta_{\ell}^{\tau}\Phi(\varphi_{\ell})\varphi_{\ell}\biggr\|_{\mathcal{N}_{K}^{\tau}}\\
			& =\|\Phi\|_{(\mathcal{N}_{K}^{\tau})'}\biggl(\sum_{\ell=0}^{L}\beta_{\ell}^{-\tau}\bigl|\beta_{\ell}^{\tau}\Phi(\varphi_{\ell})\bigr|^{2}\biggr)^{1/2},
		\end{align*}
		and thus $\|\Phi\|_{\cN_{K}^{-\tau}}=\big(\sum_{\ell=0}^{\infty}\beta_{\ell}^{\tau}|\Phi(\varphi_{\ell})|^{2}\big)^{1/2}\leq\|\Phi\|_{(\mathcal{N}_{K}^{\tau})'}<\infty$.
		Together with \eqref{eq:issubspace}, we conclude $(\cN_{K}^{\tau})'=\cN_{K}^{-\tau}$
		and $\|\cdot\|_{\cN_{K}^{-\tau}}=\|\cdot\|_{(\mathcal{N}_{K}^{\tau})'}$,
		and thus the identity operator is the sought isomorphism.
	\end{proof}
	
	\subsection{The kernel estimator}
	
	Let $Y_{1},\dots,Y_{M}\colon\Omega\to D$ be independent random variables
	that follow the distribution defined by a density $f\in\mathcal{N}_{K}$
	with respect to $\mu$, i.e.~$Y_{j}$ follows $\mathbb{P}_{Y}(A)=\int_{A}f(y)\mathrm{d}\mu(y)$,
	$A\in\mathcal{B}$. We are after an approximation to $f$ of the form
	$f(\cdot)\approx\sum_{n=1}^{N}c_{n}K(x_{n},\cdot)$ with $K(\cdot,\cdot)$
	given by \eqref{eq:K-series}, where $X=\{x_{1},\dots,x_{N}\}\subset D$
	is a set of carefully chosen points. The choice of $X$ is important,
	since given $K(\cdot,\cdot)$, it determines the approximation space
	\begin{equation}
		V_{N}:=V_{N}(X):=\spn\{K(x_{j},\cdot)\mid j=1,\dots,N\}.\label{eq:def-VN}
	\end{equation}
	
	The starting point of our method is the following ideal minimization
	problem:
	\begin{equation}
		f_{N}^{\lambda}:=\argmin_{v\in V_{N}}\tilde{J}_{\lambda}(v):=\argmin_{v\in V_{N}}\biggl[\frac{1}{2}\|v-f\|_{L_{\mu}^{2}(D)}^{2}+\frac{\lambda}{2}\|v\|_{K}^{2}\biggr].\label{eq:problem-ideal}
	\end{equation}
	This method is not practical as the evaluation of $\tilde{J}_{\lambda}$
	requires full knowledge of the target density $f$. Let us rewrite
	$\tilde{J}_{\lambda}$ as
	\begin{align*}
		\tilde{J}_{\lambda}(v) & =\frac{1}{2}\|v\|_{L_{\mu}^{2}(D)}^{2}-\langle v,f\rangle_{L_{\mu}^{2}(D)}+\frac{1}{2}\|f\|_{L_{\mu}^{2}(D)}^{2}+\frac{\lambda}{2}\|v\|_{K}^{2}\\
		& =J_{\lambda}(v)+\frac{1}{2}\|f\|_{L_{\mu}^{2}(D)}^{2},
	\end{align*}
	with $J_{\lambda}(v)=\frac{1}{2}\|v\|_{L_{\mu}^{2}(D)}^{2}-\langle v,f\rangle_{L_{\mu}^{2}(D)}+\frac{\lambda}{2}\|v\|_{K}^{2}$.
	Then, since $\frac{1}{2}\|f\|_{L_{\mu}^{2}(D)}^{2}$ is a constant
	function in $v$, we have
	\[
	f_{N}^{\lambda}=\argmin_{v\in V_{N}}\tilde{J}_{\lambda}(v)=\argmin_{v\in V_{N}}J_{\lambda}(v).
	\]
	Now we approximate $J_{\lambda}(v)$ using the i.i.d.~sample $\boldsymbol{Y}:=(Y_{1},\dots,Y_{M})\sim f\mathrm{d}\mu$,
	which yields
	\[
	f_{M,N;\boldsymbol{Y}}^{\lambda}:=\argmin_{v\in V_{N}}J_{M,\lambda}(v):=\argmin_{v\in V_{N}}\biggl[\frac{1}{2}\|v\|_{L_{\mu}^{2}(D)}^{2}+\frac{\lambda}{2}\|v\|_{K}^{2}-\frac{1}{M}\sum_{m=1}^{M}v(Y_{m})\biggr].
	\]
	This is the minimization problem we solve, which can be equivalently
	written as: Find $f_{M,N;\boldsymbol{Y}}^{\lambda}\in V_{N}$ such
	that 
	\begin{equation}
		\langle f_{M,N;\boldsymbol{Y}}^{\lambda},v\rangle_{\lambda}=\frac{1}{M}\sum_{m=1}^{M}v(Y_{m})\quad\text{for all }v\in V_{N},\label{eq:eq-disc}
	\end{equation}
	where $\langle\cdot,\cdot\rangle_{\lambda}$ is as in \eqref{eq:ip-lambda};
	see for example \cite[Theorem 6.1-1]{Ciarlet.P_2013_book_linear_nonlinear}
	for this equivalence. To see that this problem is well defined, note
	that $\Delta_{\boldsymbol{Y}}\colon\mathcal{N}_{K}\to\mathbb{R}$
	defined by
	\begin{equation}
		\Delta_{\boldsymbol{Y}(\omega)}(v):=\frac{1}{M}\sum_{m=1}^{M}v(Y_{m})\label{eq:def-Delta}
	\end{equation}
	is a linear continuous functional on $\mathcal{N}_{K}$, hence on
	$V_{N}$, for any $\omega\in\Omega$, since $\Delta_{\boldsymbol{Y}(\omega)}$
	is the sum of point evaluation functionals on the RKHS $\mathcal{N}_{K}$.
	Hence, in view of the Riesz representation theorem, the solution $f_{M,N;\boldsymbol{Y}}^{\lambda}$
	exists and is unique in $V_{N}$. The corresponding coefficients $\boldsymbol{c}=(c_{1},\dots,c_{n})^{\top}$
	satisfy the linear system
	\begin{equation}
		\mathbf{A}\boldsymbol{c}=\mathbf{b},\label{eq:eq-matrix-vector}
	\end{equation}
	where the matrix $\mathbf{A}$ is given by $\mathbf{A}_{jk}=\langle K(x_{j},\cdot),K(x_{k},\cdot)\rangle_{L_{\mu}^{2}(D)}+\lambda K(x_{j},x_{k})$,
	for $j,k=1,\dots,N$ and the vector $\mathbf{b}$ is given by $\mathbf{b}_{j}=\frac{1}{M}\sum_{m=1}^{M}K(x_{j},Y_{m}(\omega))$,
	for $j=1,\dots,N$. Notice that, as we mentioned before, the solution
	of \eqref{eq:eq-disc} exists uniquely in $V_{N}$, and thus $\mathbf{b}$
	is in the columns space of $\mathbf{A}$. Nevertheless, the equation
	\eqref{eq:eq-matrix-vector} may not be uniquely solvable if the functions
	$K(x_{j},\cdot)$, $j=1,\dots,N$ are linearly dependent. In such
	a case where $\mathbf{A}$ is singular, we take $\boldsymbol{c}\in\mathbb{R}^{N}$
	such that $\boldsymbol{c}\in(\text{null}(\mathbf{A}))^{\perp}$, where
	the orthogonal complement is taken with respect the Euclidean inner
	product.
	
	The resulting mapping $\mathbf{b}\mapsto\boldsymbol{c}$ is continuous,
	and since $\omega\mapsto\mathbf{b}(\omega)$ is $\mathcal{F}/\mathcal{B}(\mathbb{R}^{N})$-measurable,
	where $\mathcal{B}(\mathbb{R}^{N})$ is the Borel $\sigma$-algebra
	of $\mathbb{R}^{N}$, $\omega\mapsto\boldsymbol{c}(\omega)$ is also
	$\mathcal{F}/\mathcal{B}(\mathbb{R}^{N})$-measurable.
	
	Taking the expectation on both sides of \eqref{eq:eq-disc} leads
	to
	\begin{equation}
		\langle\bbE[f_{M,N;\boldsymbol{Y}}^{\lambda}],v\rangle_{\lambda}=\bbE[v(Y_{1})]=\int_{D}f(y)v(y)\mathrm{d}\mu(y)\quad\text{for all }v\in V_{N},\label{eq:eq-cont}
	\end{equation}
	and thus $f_{M,N;\boldsymbol{Y}}^{\lambda}$ is an estimator such
	that its expectation is the solution to the variational problem \eqref{eq:problem-ideal},
	i.e.~$\bbE[f_{M,N;\boldsymbol{Y}}^{\lambda}]=f_{N}^{\lambda}$.
	
	\section{General error estimate\label{sec:gen-theory}}
	
	We measure the error in terms of the mean integrated squared error
	(MISE):
	\begin{align}
		\bbE\biggl[\int_{D} & |f_{M,N;\boldsymbol{Y}}^{\lambda}(x)-f(x)|^{2}\mathrm{d}\mu(x)\biggr]\nonumber \\
		& =\|\bbE[f_{M,N;\boldsymbol{Y}}^{\lambda}]-f\|_{L_{\mu}^{2}(D)}^{2}+\bbE[\|f_{M,N;\boldsymbol{Y}}^{\lambda}-\bbE[f_{M,N;\boldsymbol{Y}}^{\lambda}]\|_{L_{\mu}^{2}(D)}^{2}].\label{eq:MISE-decomp}
	\end{align}
	In the following, we will analyse the first term (hereafter called
	the \emph{squared bias} term) and the second term (hereafter called
	the \emph{variance} term) separately.
	
	\subsection{Bias estimate}
	
	To study the bias, we introduce the $\mathcal{N}_{K}$-orthogonal
	projection $\mathscr{P}_{N}g$ of $g$ from $\mathcal{N}_{K}$ onto
	$V_{N}$ and relate the bias with the projection error. If $K$ is
	a strictly positive definite kernel, the kernel interpolation $\mathscr{I}_{N}g\in V_{N}$
	of $g\in\mathcal{N}_{K}$ that interpolates $g$ at distinct $x_{1},\dots,x_{N}$
	can be uniquely determined, and it is well known that $\mathscr{I}_{N}g=\mathscr{P}_{N}g$.
	Since a large number of interpolation error estimates are readily
	available for various kernels, this will directly provide estimates
	on the bias. For more details on the kernel interpolation, see for
	example~\cite{Wendland.H_2004_book}.
	
	Trivially, we have
	\begin{align}
		\|\bbE[f_{M,N;\boldsymbol{Y}}^{\lambda}]-f\|_{L_{\mu}^{2}(D)}^{2}=\langle\bbE[f_{M,N;\boldsymbol{Y}}^{\lambda}] & -f,\bbE[f_{M,N;\boldsymbol{Y}}^{\lambda}]-\mathscr{P}_{N}f\rangle_{L_{\mu}^{2}(D)}\nonumber \\
		& +\langle\bbE[f_{M,N;\boldsymbol{Y}}^{\lambda}]-f,\mathscr{P}_{N}f-f\rangle_{L_{\mu}^{2}(D)},\label{eq:bias-decomp}
	\end{align}
	the second term of which can be bounded as
	\begin{align}
		\langle\bbE[f_{M,N;\boldsymbol{Y}}^{\lambda}]&-f,\mathscr{P}_{N}f-f\rangle_{L_{\mu}^{2}(D)}\notag\\
		& \leq\frac{1}{2}\|\bbE[f_{M,N;\boldsymbol{Y}}^{\lambda}]-f\|_{L_{\mu}^{2}(D)}^{2}+\frac{1}{2}\|\mathscr{P}_{N}f-f\|_{L_{\mu}^{2}(D)}^{2}.\label{eq:bias-2nd}
	\end{align}
	Bounding the first term is more involved.
	\begin{lemma}
		\label{lem:1st-bias}Let $\{x_{1},\dots,x_{N}\}\subset D$ be arbitrary
		and let $V_{N}$ be the corresponding space~\eqref{eq:def-VN}. Then,
		for the $\mathcal{N}_{K}$-orthogonal projection $\mathscr{P}_{N}\colon\mathcal{N}_{K}\to V_{N}$,
		we have 
		\[
		\langle\bbE[f_{M,N;\boldsymbol{Y}}^{\lambda}]-f,\bbE[f_{M,N;\boldsymbol{Y}}^{\lambda}]-\mathscr{P}_{N}f\rangle_{L_{\mu}^{2}(D)}\leq\frac{1}{2}\lambda\|f\|_{K}^{2}.
		\]
	\end{lemma}
	
	\begin{proof}
		From $\bbE[f_{M,N;\boldsymbol{Y}}^{\lambda}]-\mathscr{P}_{N}f\in V_{N}$, 
		the equation \eqref{eq:eq-cont} implies
		\begin{align*}
			\langle\bbE[&f_{M,N;\boldsymbol{Y}}^{\lambda}],\bbE[f_{M,N;\boldsymbol{Y}}^{\lambda}]-\mathscr{P}_{N}f\rangle_{L_{\mu}^{2}(D)} 
			\\
			+&\lambda\langle\bbE[f_{M,N;\boldsymbol{Y}}^{\lambda}]-f+f,\bbE[f_{M,N;\boldsymbol{Y}}^{\lambda}]-\mathscr{P}_{N}f\rangle_{K} =\langle f,\bbE[f_{M,N;\boldsymbol{Y}}^{\lambda}]-\mathscr{P}_{N}f\rangle_{L_{\mu}^{2}(D)},
		\end{align*}
		and thus
		\begin{align}
			\langle\bbE[&f_{M,N;\boldsymbol{Y}}^{\lambda}]-f,\bbE[f_{M,N;\boldsymbol{Y}}^{\lambda}]-\mathscr{P}_{N}f\rangle_{L_{\mu}^{2}(D)}\notag\\
			&+  \lambda\langle\bbE[f_{M,N;\boldsymbol{Y}}^{\lambda}]-f,\bbE[f_{M,N;\boldsymbol{Y}}^{\lambda}]-\mathscr{P}_{N}f\rangle_{K}  =-\lambda\langle f,\bbE[f_{M,N;\boldsymbol{Y}}^{\lambda}]-\mathscr{P}_{N}f\rangle_{K}.\label{eq:before-Young}
		\end{align}
		Since $\mathscr{P}_{N}$ is the $\mathcal{N}_{K}$-orthogonal projection,
		we have $\langle\bbE[f_{M,N;\boldsymbol{Y}}^{\lambda}]-f,\bbE[f_{M,N;\boldsymbol{Y}}^{\lambda}]-\mathscr{P}_{N}f\rangle_{K}=\|\bbE[f_{M,N;\boldsymbol{Y}}^{\lambda}]-\mathscr{P}_{N}f\|_{K}^{2}$,
		and thus the Young's inequality implies
		\[
		\langle\bbE[f_{M,N;\boldsymbol{Y}}^{\lambda}]-f,\bbE[f_{M,N;\boldsymbol{Y}}^{\lambda}]-\mathscr{P}_{N}f\rangle_{L_{\mu}^{2}(D)}+\frac{1}{2}\lambda\|\bbE[f_{M,N;\boldsymbol{Y}}^{\lambda}]-\mathscr{P}_{N}f\|_{K}^{2}\leq\frac{1}{2}\lambda\|f\|_{K}^{2}.
		\]
		This completes the proof.
	\end{proof}
	Hence, we obtain an estimate of the squared bias.
	\begin{proposition}
		\label{prop:sq-bias}Let the assumptions of Lemma~\ref{lem:1st-bias}
		hold. Then, the solution $f_{M,N;\boldsymbol{Y}}^{\lambda}$ to \eqref{eq:eq-disc}
		satisfies
		\[
		\|\bbE[f_{M,N;\boldsymbol{Y}}^{\lambda}]-f\|_{L_{\mu}^{2}(D)}^{2}\leq\|\mathscr{P}_{N}f-f\|_{L_{\mu}^{2}(D)}^{2}+\lambda\|f\|_{K}^{2}.
		\]
	\end{proposition}
	
	By exploiting a stronger smoothness of $f$, we can establish a bound
	that is of second order in $\lambda$.
	\begin{proposition}
		Let $\{x_{1},\dots,x_{N}\}\subset D$ be arbitrary and let $V_{N}$
		be the corresponding space~\eqref{eq:def-VN}. Suppose $f\in\mathcal{N}_{K}^{2}$.
		Then, the solution $f_{M,N;\boldsymbol{Y}}^{\lambda}$ to \eqref{eq:eq-disc}
		satisfies
		\[
		\|\bbE[f_{M,N;\boldsymbol{Y}}^{\lambda}]-f\|_{L_{\mu}^{2}(D)}^{2}\leq3\|\mathscr{P}_{N}f-f\|_{L_{\mu}^{2}(D)}^{2}+8\lambda^{2}\|f\|_{\mathcal{N}_{K}^{2}}^{2}.
		\]
	\end{proposition}
	
	\begin{proof}
		From \eqref{eq:before-Young} we have
		\begin{align*}
			\langle\bbE[&f_{M,N;\boldsymbol{Y}}^{\lambda}]-f,\bbE[f_{M,N;\boldsymbol{Y}}^{\lambda}]  -\mathscr{P}_{N}f\rangle_{L_{\mu}^{2}(D)}+\lambda\|\bbE[f_{M,N;\boldsymbol{Y}}^{\lambda}]-\mathscr{P}_{N}f\|_{K}^{2}\\
			& =-\lambda\langle f,\bbE[f_{M,N;\boldsymbol{Y}}^{\lambda}]-\mathscr{P}_{N}f\rangle_{K}\\
			& \leq\lambda\|f\|_{\mathcal{N}_{K}^{2}}\|\bbE[f_{M,N;\boldsymbol{Y}}^{\lambda}]-\mathscr{P}_{N}f\|_{L_{\mu}^{2}(D)}\\
			& \leq2\lambda^{2}\|f\|_{\mathcal{N}_{K}^{2}}^{2}+\frac{1}{8}\|\bbE[f_{M,N;\boldsymbol{Y}}^{\lambda}]-\mathscr{P}_{N}f\|_{L_{\mu}^{2}(D)}^{2}\\
			& \leq2\lambda^{2}\|f\|_{\mathcal{N}_{K}^{2}}^{2}+\frac{1}{4}\|\bbE[f_{M,N;\boldsymbol{Y}}^{\lambda}]-f\|_{L_{\mu}^{2}(D)}^{2}+\frac{1}{4}\|f-\mathscr{P}_{N}f\|_{L_{\mu}^{2}(D)}^{2}.
		\end{align*}
		Thus, \eqref{eq:bias-decomp} and \eqref{eq:bias-2nd} imply
		\[
		\Bigl(\frac{1}{2}-\frac{1}{4}\Bigr)\|\bbE[f_{M,N;\boldsymbol{Y}}^{\lambda}]-f\|_{L_{\mu}^{2}(D)}^{2}\leq\Bigl(\frac{1}{2}+\frac{1}{4}\Bigr)\|\mathscr{P}_{N}f-f\|_{L_{\mu}^{2}(D)}^{2}+2\lambda^{2}\|f\|_{\mathcal{N}_{K}^{2}}^{2}.
		\]
		Now the proof is complete.
	\end{proof}
	
	\subsection{Variance estimate}
	
	Now we bound the variance $\bbE[\|f_{M,N;\boldsymbol{Y}}^{\lambda}-\bbE[f_{M,N;\boldsymbol{Y}}^{\lambda}]\|_{L_{\mu}^{2}(D)}^{2}]$
	in \eqref{eq:MISE-decomp}. Taking the difference of~\eqref{eq:eq-disc}
	and \eqref{eq:eq-cont} yields\begin{equation}
		\langle f_{M,N;\boldsymbol{Y}}^{\lambda}-\bbE[f_{M,N;\boldsymbol{Y}}^{\lambda}],v\rangle_{\lambda}=\Delta_{\boldsymbol{Y}}(v)-F(v)\quad\text{for all }v\in V_{N},\label{eq:eq-for-var}
	\end{equation}
	where $\Delta_{\boldsymbol{Y}}$ is defined in \eqref{eq:def-Delta},
	and we used the notation
	\begin{equation}
		F(v):=\int_{D}v(y)f(y)\mathrm{d}\mu(y).\label{eq:def-F}
	\end{equation}
	Then, we have $F\in(\mathcal{N}_{K})'$, and as point-evaluation functionals
	are continuous on $\mathcal{N}_{K}$, we also have $\Delta_{\boldsymbol{Y}}\in(\mathcal{N}_{K})'$.
	
	If $f$ is smooth and thus accordingly $K$ is taken to be smooth,
	the corresponding space $\mathcal{N}_{K}$ may be smaller than necessary
	for $\Delta_{\boldsymbol{Y}}$ to be continuous. Namely, in such cases
	$\Delta_{\boldsymbol{Y}}$ is continuous on larger spaces $\mathcal{N}_{K}^{\tau}$,
	$\tau\in(\tau_{0},1)$ for some $\tau_{0}\geq0$. We will exploit
	this observation in the variance estimate and assess the variance
	of $\Delta_{\boldsymbol{Y}}-F$ in $\cN_{K}^{-\tau}$.
	\begin{lemma}
		Suppose that for some $\tau\in(0,1]$ we have $\Delta_{\boldsymbol{Y}}\in\cN_{K}^{-\tau}$,
		and that $f$ satisfies $\langle K_{\tau}(\cdot,\cdot),f\rangle_{L_{\mu}^{2}(D)}:=\int_{D}K_{\tau}(x,x)f(x)\,\mathrm{d}\mu(x)<\infty$
		with $K_{\tau}(x_{1},x_{2}):=\sum_{\ell=0}^{\infty}\beta_{\ell}^{\tau}\varphi_{\ell}(x_{1})\varphi_{\ell}(x_{2})$.
		Then, the equality
		\[
		\mathbb{E}[\|\Delta_{\boldsymbol{Y}}-F\|_{\cN_{K}^{-\tau}}^{2}]=\frac{\langle K_{\tau}(\cdot,\cdot),f\rangle_{L_{\mu}^{2}(D)}-\|F\|_{\cN_{K}^{-\tau}}^{2}}{M}
		\]
		holds, where $F$ is defined in~\eqref{eq:def-F}.
	\end{lemma}
	
	\begin{proof}
		We have
		\begin{equation}
			\mathbb{E}[\|\Delta_{\boldsymbol{Y}}-F\|_{\cN_{K}^{-\tau}}^{2}]=\mathbb{E}[\|\Delta_{\boldsymbol{Y}}\|_{\cN_{K}^{-\tau}}^{2}]-2\mathbb{E}[\langle\Delta_{\boldsymbol{Y}},F\rangle_{\cN_{K}^{-\tau}}]+\|F\|_{\cN_{K}^{-\tau}}^{2}.\label{eq:expand-var}
		\end{equation}
		For the second term, first notice that for any $L\in\mathbb{N}$ we
		have
		\begin{align*}
			\mathbb{E}\biggl[\biggl|\sum_{\ell=0}^{L}\beta_{\ell}^{\tau}\Delta_{\boldsymbol{Y}}(\varphi_{\ell})F(\varphi_{\ell})\biggr|\biggr] & \leq\mathbb{E}\biggl[\frac{1}{M}\sum_{m=1}^{M}\sum_{\ell=0}^{L}\beta_{\ell}^{\tau}|\varphi_{\ell}(Y_{m})|\,|\langle f,\varphi_{\ell}\rangle_{L_{\mu}^{2}(D)}|\biggr]\\
			& =\frac{1}{M}\sum_{m=1}^{M}\sum_{\ell=0}^{L}\beta_{\ell}^{\tau}\mathbb{E}[|\varphi_{\ell}(Y_{m})|]\,|\langle f,\varphi_{\ell}\rangle_{L_{\mu}^{2}(D)}|\\
			& \leq\frac{1}{M}\sum_{m=1}^{M}\sum_{\ell=0}^{L}\beta_{\ell}^{\tau}\sqrt{\mathbb{E}[|\varphi_{\ell}(Y_{m})|^{2}]}\,|\langle f,\varphi_{\ell}\rangle_{L_{\mu}^{2}(D)}|\\
			& =\sum_{\ell=0}^{L}\beta_{\ell}^{\tau}\sqrt{\langle\varphi_{\ell}^{2},f\rangle_{L_{\mu}^{2}(D)}}|\langle f,\varphi_{\ell}\rangle_{L_{\mu}^{2}(D)}|\\
			& \leq\sqrt{\sum_{\ell=0}^{L}\langle\beta_{\ell}^{\tau}\varphi_{\ell}^{2},f\rangle_{L_{\mu}^{2}(D)}}\sqrt{\sum_{\ell=0}^{L}\beta_{\ell}^{\tau}|\langle f,\varphi_{\ell}\rangle_{L_{\mu}^{2}(D)}|^{2}}\\
			& \leq\sqrt{\langle K_{\tau}(\cdot,\cdot),f\rangle_{L_{\mu}^{2}(D)}}\|F\|_{\cN_{K}^{-\tau}}<\infty,
		\end{align*}
		where in the last equality we used the non-negativity of $\beta_{\ell}^{\tau}\varphi_{\ell}^{2}f$,
		and in view of Proposition~\ref{prop:dual-isom} the slight abuse
		of notation $\|F\|_{\cN_{K}^{-\tau}}$ should be unambiguous. Hence,
		we can use the dominated convergence theorem to conclude
		\[
		-2\mathbb{E}[\langle\Delta_{\boldsymbol{Y}},F\rangle_{\cN_{K}^{-\tau}}]=-2\sum_{\ell=0}^{\infty}\beta_{\ell}\mathbb{E}[\Delta_{\boldsymbol{Y}}(\varphi_{\ell})]F(\varphi_{\ell})=-2\|F\|_{\cN_{K}^{-\tau}}.
		\]
		The first term in the right-hand side of~\eqref{eq:expand-var} can
		be rewritten as
		\begin{align*}
			\mathbb{E}[\|\Delta_{\boldsymbol{Y}}\|_{\cN_{K}^{-\tau}}^{2}] & =\frac{1}{M^{2}}\sum_{m=1}^{M}\mathbb{E}[\|\delta_{Y_{m}}\|_{\cN_{K}^{-\tau}}^{2}]+\frac{1}{M^{2}}\sum_{m=1}^{M}\sum_{\substack{k=1\\
					k\neq m
				}
			}^{M}\mathbb{E}[\langle\delta_{Y_{m}},\delta_{Y_{k}}\rangle_{\cN_{K}^{-\tau}}]\\
			& =\frac{1}{M}\sum_{\ell=0}^{\infty}\beta_{\ell}^{\tau}\mathbb{E}[\varphi_{\ell}(Y_{1})^{2}]+\frac{1}{M^{2}}\sum_{m=1}^{M}\sum_{\substack{k=1\\
					k\neq m
				}
			}^{M}\sum_{\ell=0}^{\infty}\beta_{\ell}^{\tau}\mathbb{E}[\varphi_{\ell}(Y_{m})]\mathbb{E}[\varphi_{\ell}(Y_{k})]\\
			& =\frac{\langle K_{\tau}(\cdot,\cdot),f\rangle_{L_{\mu}^{2}(D)}}{M}+\frac{M-1}{M}\|F\|_{\cN_{K}^{-\tau}}^{2},
		\end{align*}
		where we used the notation $\delta_{Y_{m}}(v):=v(Y_{m})$ for $m=1,\dots,M$.
		Hence, we conclude
		\[
		\mathbb{E}[\|\Delta_{\boldsymbol{Y}}-F\|_{\cN_{K}^{-\tau}}^{2}]=\frac{\langle K_{\tau}(\cdot,\cdot),f\rangle_{L_{\mu}^{2}(D)}-\|F\|_{\cN_{K}^{-\tau}}^{2}}{M}.
		\]
	\end{proof}
	We arrive at the variance estimate of our density approximation. The
	proof is inspired by \cite[Theorem 2]{Arnone.E_Kneip_Nobile_Sangalli_2022_regression}.
	\begin{proposition}
		\label{prop:var-estim}Suppose that for some $\tau\in(0,1]$ we have
		$\Delta_{\boldsymbol{Y}}\in\cN_{K}^{-\tau}$, and that $f$ satisfies
		$\langle K_{\tau}(\cdot,\cdot),f\rangle_{L_{\mu}^{2}(D)}<\infty$
		with $K_{\tau}(x_{1},x_{2})=\sum_{\ell=0}^{\infty}\beta_{\ell}^{\tau}\varphi_{\ell}(x_{1})\varphi_{\ell}(x_{2})$.
		Then, for any $\lambda\in(0,1]$ we have
		\[
		\mathbb{E}[\|f_{M,N;\boldsymbol{Y}}^{\lambda}-\bbE[f_{M,N;\boldsymbol{Y}}^{\lambda}]\|_{\lambda}^{2}]\leq\frac{\langle K_{\tau}(\cdot,\cdot),f\rangle_{L_{\mu}^{2}(D)}}{M\lambda^{\tau}}.
		\]
	\end{proposition}
	
	\begin{proof}
		For $\tau\in(0,1]$ such that $\Delta_{\boldsymbol{Y}}\in\cN_{K}^{-\tau}$,
		\rv{choosing $v=f_{M,N;\boldsymbol{Y}}^{\lambda}-\bbE[f_{M,N;\boldsymbol{Y}}^{\lambda}]$ in~\eqref{eq:eq-for-var} and taking the expectation on both sides yields}
		\begin{align*}
			\mathbb{E}[\|f_{M,N;\boldsymbol{Y}}^{\lambda}-\bbE[f_{M,N;\boldsymbol{Y}}^{\lambda}] & \|_{\lambda}^{2}]\\
			\leq & \lambda^{-\tau/2}\mathbb{E}[\|\Delta_{\boldsymbol{Y}}-F\|_{\cN_{K}^{-\tau}}\lambda^{\tau/2}\|f_{M,N;\boldsymbol{Y}}^{\lambda}-\bbE[f_{M,N;\boldsymbol{Y}}^{\lambda}]\|_{\cN_{K}^{\tau}}].
		\end{align*}
		For any $v\in\cN_{K}$, from $1/\tau\in[1,\infty)$ we have
		\begin{align*}
			\lambda^{\tau}\|v\|_{\cN_{K}^{\tau}}^{2} & =\lambda^{\tau}\sum_{\ell=0}^{\infty}\frac{1}{\beta_{\ell}^{\tau}}\langle v,\varphi_{\ell}\rangle_{L_{\mu}^{2}(D)}^{2(\tau+(1-\tau))}\\
			& \leq\lambda^{\tau}\Biggl(\sum_{\ell=0}^{\infty}\frac{1}{\beta_{\ell}^{\tau\frac{1}{\tau}}}\langle v,\varphi_{\ell}\rangle_{L_{\mu}^{2}(D)}^{2}\Biggr)^{\tau}\Biggl(\sum_{\ell=0}^{\infty}\langle v,\varphi_{\ell}\rangle_{L_{\mu}^{2}(D)}^{2}\Biggr)^{1-\tau}\\
			& =\lambda^{\tau}\|v\|_{\cN_{K}}^{2\tau}\|v\|_{L_{\mu}^{2}(D)}^{2(1-\tau)}\leq\tau\lambda\|v\|_{\cN_{K}}^{2}+(1-\tau)\|v\|_{L_{\mu}^{2}(D)}^{2}\leq\|v\|_{\lambda}^{2},
		\end{align*}
		and thus 
		\[
		\|v\|_{\cN_{K}^{\tau}}\leq\lambda^{-\tau/2}\|v\|_{\lambda}.
		\]
		Hence, we obtain
		\begin{align*}
			\mathbb{E}[\|f_{M,N;\boldsymbol{Y}}^{\lambda}&-\bbE[f_{M,N;\boldsymbol{Y}}^{\lambda}]\|_{\lambda}^{2}]  \leq\lambda^{-\tau/2}\mathbb{E}[\|\Delta_{\boldsymbol{Y}}-F\|_{\cN_{K}^{-\tau}}\|f_{M,N;\boldsymbol{Y}}^{\lambda}-\bbE[f_{M,N;\boldsymbol{Y}}^{\lambda}]\|_{\lambda}]\\
			& \leq\lambda^{-\tau/2}\sqrt{\frac{\langle K_{\tau}(\cdot,\cdot),f\rangle_{L_{\mu}^{2}(D)}}{M}}\sqrt{\mathbb{E}[\|f_{M,N;\boldsymbol{Y}}^{\lambda}-\bbE[f_{M,N;\boldsymbol{Y}}^{\lambda}]\|_{\lambda}^{2}]},
		\end{align*}
		and thus
		\[
		\mathbb{E}[\|f_{M,N;\boldsymbol{Y}}^{\lambda}-\bbE[f_{M,N;\boldsymbol{Y}}^{\lambda}]\|_{\lambda}^{2}]\leq\frac{\langle K_{\tau}(\cdot,\cdot),f\rangle_{L_{\mu}^{2}(D)}}{M\lambda^{\tau}}.
		\]
		The proof is now complete.
	\end{proof}
	
	\subsection{MISE estimate}
	
	We summarise the discussions so far as a theorem.
	\begin{theorem}
		\label{thm:general-summary}Let $f\in\mathcal{N}_{K}$ be the target
		density function and let $f_{M,N;\boldsymbol{Y}}^{\lambda}\in V_{N}$
		satisfy~\eqref{eq:eq-disc}. Moreover, let $\mathscr{P}_{N}\colon\mathcal{N}_{K}\to V_{N}$
		be the $\mathcal{N}_{K}$-orthogonal projection. Suppose that for
		some $\tau\in(0,1]$ we have $\Delta_{\boldsymbol{Y}}\in\cN_{K}^{-\tau}$,
		and that $f$ satisfies $\langle K_{\tau}(\cdot,\cdot),f\rangle_{L_{\mu}^{2}(D)}<\infty$
		with $K_{\tau}(x_{1},x_{2})=\sum_{\ell=0}^{\infty}\beta_{\ell}^{\tau}\varphi_{\ell}(x_{1})\varphi_{\ell}(x_{2})$.
		Then, we have the MISE estimate
		\begin{equation}
			\bbE\Bigl[\!\,\int_{D}\!
			|f_{M,N;\boldsymbol{Y}}^{\lambda}(x)-f(x)|^{2}\mathrm{d}\mu(x)\!\,\Bigr]
			\!\leq\!
			\|\mathscr{P}_{N}f-f\|_{\!\,L_{\mu}^{2}(D)}^{2}+\lambda\|f\|_{K}^{2}+\frac{\langle K_{\tau}(\cdot,\cdot),f\rangle_{L_{\mu}^{2}(D)}}{M\lambda^{\tau}}.\label{eq:general-result}
		\end{equation}
		Suppose furthermore $f\in\mathcal{N}_{K}^{2}$. Then we also have
		\begin{align}
			\bbE\Bigl[\int_{D}\!&|f_{M,N;\boldsymbol{Y}}^{\lambda}(x)-f(x)|^{2}\mathrm{d}\mu(x)\Bigr]\notag\\
			&\leq
			3\|\mathscr{P}_{N}f-f\|_{L_{\mu}^{2}(D)}^{2}+8\lambda^{2}\|f\|_{\mathcal{N}_{K}^{2}}^{2}+\frac{\langle K_{\tau}(\cdot,\cdot),f\rangle_{L_{\mu}^{2}(D)}}{M\lambda^{\tau}}.\label{eq:general-result-2}
		\end{align}
	\end{theorem}
	
	\subsection{Limiting case \texorpdfstring{$N\to\infty$}{} and link with kernel density estimation}
	
	In principle, the projection error in the right hand side of \eqref{eq:general-result}
	and \eqref{eq:general-result-2} is independent of the other two terms.
	Thus, it may be natural to pose the problem in $\mathcal{N}_{K}$
	rather than in $V_{N}$: Find $f_{M;\boldsymbol{Y}}^{\lambda}\in\mathcal{N}_{K}$
	such that 
	\[
	\langle f_{M;\boldsymbol{Y}}^{\lambda},v\rangle_{\lambda}=\frac{1}{M}\sum_{m=1}^{M}v(Y_{m})\quad\text{for all }v\in\mathcal{N}_{K}.
	\]
	Then, from
	\begin{align*}
		\langle f_{M;\boldsymbol{Y}}^{\lambda},\varphi_{\ell}\rangle_{L_{\mu}^{2}(D)}+\frac{\lambda}{\beta_{\ell}}\langle f_{M;\boldsymbol{Y}}^{\lambda},\varphi_{\ell}\rangle_{L_{\mu}^{2}(D)} & =\frac{1}{M}\sum_{m=1}^{M}\varphi_{\ell}(Y_{m})\text{\ensuremath{\quad} for }\ell\geq0,
	\end{align*}
	the solution is given by a linear combination 
	\[
	f_{M;\boldsymbol{Y}}^{\lambda}=\frac{1}{M}\sum_{m=1}^{M}K_{\lambda}^{*}(Y_{m},\cdot)
	\]
	of the kernel $K_{\lambda}^{*}(x,x'):=\sum_{\ell=0}^{\infty}\frac{\beta_{\ell}}{\beta_{\ell}+\lambda}\varphi_{\ell}(x)\varphi_{\ell}(x')$,
	$x,x'\in D$, which is of the form similar to the standard kernel
	density estimation; see for example \cite{Scott.D_2015_MultivariateDensityEstimation}.
	Note that we indeed have $f_{M;\boldsymbol{Y}}^{\lambda}\in\mathcal{N}_{K}$:
	from
	\[
	\frac{\beta_{\ell}}{\beta_{\ell}+\lambda}\langle\varphi_{\ell}(Y_{m}(\omega))\varphi_{\ell},\varphi_{\ell}\rangle_{L_{\mu}^{2}(D)}=\frac{\beta_{\ell}}{\beta_{\ell}+\lambda}\varphi_{\ell}(Y_{m}(\omega))
	\]
	provided $\lambda>0$, for any realization $\boldsymbol{y}=(y_{1},\dots,y_{M})=\boldsymbol{Y}(\omega)$
	we have
	\begin{align*}
		\|f_{M;\boldsymbol{y}}^{\lambda}\|_{K}^{2} & =\sum_{\ell=0}^{\infty}\frac{\beta_{\ell}}{(\beta_{\ell}+\lambda)^{2}}\frac{1}{M}\sum_{m=1}^{M}\varphi_{\ell}(y_{m})^{2}\leq\frac{1}{\lambda^{2}}\frac{1}{M}\sum_{m=1}^{M}K(y_{m},y_{m})<\infty.
	\end{align*}
	Notice however that $f_{M;\boldsymbol{Y}}^{\lambda}$ is not a linear
	combination of the kernel $K$, and in general the kernel $K_{\lambda}^{*}$
	cannot be given in a closed form, even if $K$ is. Hence, in practice
	it is much more computationally efficient to seek the approximation
	in $V_{N}$.
	
	The problem posed in $\mathcal{N}_{K}$ in the discussion above can
	be seen as a limiting case of our finite dimensional setting with
	$N\to\infty$ and a dense subset $X=\{x_{j}\}_{j\in\mathbb{N}}$ of
	$D$ in the following sense. Suppose that $D$ is a separable metric
	space, $\mathcal{B}$ is a corresponding Borel $\sigma$-algebra,
	and that $\mu$ is a $\sigma$-finite measure on $(D,\mathcal{B})$.
	Then, the Hilbert space of equivalence classes of square integrable
	functions $L_{\mu}^{2}(D)$ is separable; see for example~\cite[p. 92]{Doob.J.L_1994_Book_MeasureTheory}.
	Moreover, assume that the positive definite kernel $K\colon D\times D\to\mathbb{R}$
	is continuous. Then, there exists a dense subset $\{q_{j}\}_{j\in\mathbb{N}}\subset D$
	such that $\overline{\mathrm{span}\{\{K(q_{j},\cdot)\}_{j\in\mathbb{N}}\}}=\mathcal{N}_{K}$,
	in particular $\mathcal{N}_{K}$ is separable as shown in the next
	proposition.
	\begin{proposition}
		Under the assumptions above on $(D,\mathcal{B},\mu)$, there exists
		a subset $\{q_{j}\}_{j\in\mathbb{N}}\subset D$ such that $\overline{\mathrm{span}\{\{K(q_{j},\cdot)\}_{j\in\mathbb{N}}\}}=\mathcal{N}_{K}$,
		where the closure is taken with respect to the $\mathcal{N}_{K}$-norm.
	\end{proposition}
	
	\begin{proof}
		Since $D$ is assumed to be separable, there exists a subset $\{q_{j}\}_{j\in\mathbb{N}}\subset D$
		that approximates all $x\in D$. Consider the subset $\{K(q_{j},\cdot)\}_{j\in\mathbb{N}}\subset\mathcal{N}_{K}$.
		Let $\mathcal{M}:=\overline{\mathrm{span}\{\{K(q_{j},\cdot)\}_{j\in\mathbb{N}}\}}\subset\mathcal{N}_{K}$.
		Since $\mathcal{M}$ is a closed subspace of $\mathcal{N}_{K}$, we
		have $\mathcal{N}_{K}=\mathcal{M}\oplus\mathcal{M}^{\perp}$. We will
		show $\mathcal{M}^{\perp}=\{0\}$. Indeed, for $g\in\mathcal{M}^{\perp}$,
		we have
		\[
		0=\langle g,K(q_{j},\cdot)\rangle_{K}=g(q_{j})\qquad\text{for any }j\in\mathbb{N}.
		\]
		Now, observe that the continuity of the kernel implies that all elements
		in $\mathcal{N}_{K}$ are (sequentially) continuous on $D$. Therefore,
		we have $g(x)=0$ for all $x\in D$. Hence, $\mathcal{M}^{\perp}=\{0\}$.
	\end{proof}
	
	\section{Application to densities in weighted Korobov spaces\label{sec:Example}}
	
	In this section, we will apply the theory established in Section~\ref{sec:gen-theory}
	to the case where the kernel defines the so-called Korobov space.
	Throughout this section, we assume that the density function $f$
	is defined on the $d$-dimensional unit hypercube $[0,1]^{d}\subset\mathbb{R}^{d}$,
	and that the density is with respect to the uniform measure. 
	\rv{This choice of reference measure is due to the definition of the Korobov space, whose norm is based on the standard $L^2$-inner product with respect to the uniform measure.}  
	Since
	in this section $D=[0,1]^{d}$ is a subset of $\mathbb{R}^{d}$, we
	will use the bold symbol $\boldsymbol{x}$ to denote a point in $[0,1]^{d}$.
	
	Let a smoothness parameter $\alpha>1$ be given. For non-negative parameters
	$\boldsymbol{\gamma}=\{\gamma_{\mathfrak{u}}\}_{\mathfrak{u}\subset\mathbb{N}}$,
	which we call \emph{weights}, we consider the Korobov kernel
	\begin{equation}
		K_{\alpha}^{\mathrm{kor}}(\boldsymbol{x},\boldsymbol{x}')=\sum_{\boldsymbol{h}\in\mathbb{Z}^{d}}r(\boldsymbol{h},\gamma)^{-1}\mathrm{e}^{2\pi i\boldsymbol{h}\cdot(\boldsymbol{x}-\boldsymbol{x}')},\quad\boldsymbol{x},\boldsymbol{x}'\in[0,1]^{d},\label{eq:series-kor}
	\end{equation}
	where $\boldsymbol{h}\cdot\boldsymbol{x}$ denotes the Euclidean inner
	product $\boldsymbol{h}\cdot\boldsymbol{x}=\sum_{j=1}^{d}h_{j}x_{j}$,
	and
	\[
	r(\boldsymbol{h},\gamma):=\left\{ \begin{array}{ll}
		1, & \text{ if }\boldsymbol{h}=(0,\ldots,0)\\
		\gamma_{\supp(\boldsymbol{h})}^{-1}\prod_{\supp(\boldsymbol{h})}\left|h_{j}\right|^{\alpha}, & \text{ if }\boldsymbol{h}\neq(0,\ldots,0)
	\end{array},\right.
	\]
	with $\supp\boldsymbol{h}:=\{1\leq j\leq d\mid h_{j}\neq0\}$. We
	take $\gamma_{\emptyset}:=1$, so that the norm of a constant function
	in the corresponding reproducing kernel Hilbert space matches its
	$L^{2}$ norm. We denote the corresponding reproducing kernel Hilbert
	space by $\mathcal{N}_{\mathrm{kor},\alpha}$, which consists of $1$-periodic
	functions on $\mathbb{R}^{d}$ with a suitable smoothness governed
	by the parameter $\alpha>1$. This kernel can be rewritten as
	\[
	K_{\alpha}^{\mathrm{kor}}(\boldsymbol{x},\boldsymbol{x}')=1+\sum_{\emptyset\neq\mathfrak{u}\subseteq\{1,\ldots,d\}}\gamma_{\mathfrak{u}}K_{\alpha}^{\mathrm{kor}}(\boldsymbol{x}_{\mathfrak{u}},\boldsymbol{x}'_{\mathfrak{u}}),
	\]
	with
	\[
	K_{\alpha}^{\mathrm{kor}}(\boldsymbol{x}_{\mathfrak{u}},\boldsymbol{x}'_{\mathfrak{u}})=\prod_{j\in\mathfrak{u}}\left(\sum_{h\in\mathbb{Z}\setminus\{0\}}\frac{\mathrm{e}^{2\pi ih(x_{j}-x'_{j})}}{|h|^{\alpha}}\right)=\prod_{j\in\mathfrak{u}}2\sum_{h=1}^{\infty}\frac{\cos(2\pi h(x_{j}-x'_{j}))}{h^{\alpha}}.
	\]
	\sloppy{For $\alpha>1$, the kernel $K_{\alpha}^{\mathrm{kor}}$ is well defined
		for all $\boldsymbol{x},\boldsymbol{x}'\in[0,1]^{d}$ and satisfies
		$\sup_{\boldsymbol{x},\boldsymbol{x}'\in[0,1]^{d}}K_{\alpha}^{\mathrm{kor}}(\boldsymbol{x},\boldsymbol{x}')<\infty$, and with}
	\[\langle v,e^{2\pi i\boldsymbol{h}\cdot\textrm{-}}\rangle_{L^{2}([0,1]^{d})}:=\int_{[0,1]^{d}}v(\boldsymbol{x})e^{-2\pi i\boldsymbol{h}\cdot\boldsymbol{x}}\mathrm{d}\boldsymbol{x},\]
	the corresponding norm is given by
	\[
	\|v\|_{\mathrm{kor},\alpha}=\biggl(\sum_{\boldsymbol{h}\in\mathbb{Z}^{d}}r(\boldsymbol{h},\gamma)|\langle v,e^{2\pi i\boldsymbol{h}\cdot\textrm{-}}\rangle_{L^{2}([0,1]^{d})}|^{2}\biggr)^{1/2},
	\]
	where the series is absolutely convergent.
	
	If $\alpha$ is an even integer, the expression of the kernel and
	the norm simplifies. First, the reproducing kernel $K_{\alpha}^{\mathrm{kor}}(\boldsymbol{x},\boldsymbol{x}')$
	is related to the Bernoulli polynomials $B_{\alpha}$: for $\alpha$
	even, we have
	\[
	B_{\alpha}(x)=\frac{(-1)^{\frac{\alpha}{2}+1}\alpha!}{(2\pi)^{\alpha}}\sum_{h\in\mathbb{Z}\setminus\{0\}}\frac{e^{2\pi ihx}}{|h|^{\alpha}}\quad\text{for any }\ x\in[0,1],
	\]
	so that
	\[
	K_{\alpha}^{\mathrm{kor}}(\boldsymbol{x},\boldsymbol{x}')=1+\sum_{\emptyset\neq\mathfrak{u}\subseteq\{1,\ldots,d\}}\gamma_{\mathfrak{u}}\left(\frac{(2\pi)^{\alpha}}{(-1)^{\frac{\alpha}{2}+1}\alpha!}\right)^{|u|}\prod_{j\in\mathfrak{u}}B_{\alpha}(|\{x_{j}-x'_{j}\}|),
	\]
	where $\{x\}$ denotes the fractional part of $x$. Moreover, the
	norm $\|\cdot\|_{\mathrm{kor},\alpha}$ can be rewritten as the norm
	in an ``unanchored'' weighted Sobolev space of dominating mixed
	smoothness of order $\alpha/2$,
	\begin{align*}
		&\|v\|_{\mathrm{kor},\alpha}\\
		&=\sqrt{\sum_{\mathfrak{u}\subseteq\{1,\ldots,d\}}\frac{1}{(2\pi)^{\alpha|\mathfrak{u}|}\gamma_{\mathfrak{u}}}\int_{[0,1]^{|\mathfrak{u}|}}\!\!\;\bigg|\int_{[0,1]^{d-|\mathfrak{u}|}}\bigg(\prod_{j\in\mathfrak{u}}\frac{\partial^{\alpha/2}}{\partial x_{j}^{\alpha/2}}\bigg)v(\bsx)\,\mathrm{d}\bsx_{-\mathfrak{u}}\bigg|^{2}\,\mathrm{d}\bsx_{\mathfrak{u}}},
	\end{align*}
	where $\bsx_{\mathfrak{u}}$ denotes the components of $\bsx$ with
	indices that belong to the subset $\mathfrak{u}$, and $\bsx_{-\mathfrak{u}}$
	denotes the components that do not belong to $\mathfrak{u}$, and
	$|\mathfrak{u}|$ denotes the cardinality of $\mathfrak{u}$. See
	for example \cite{Novak.E_Wozniakowski_book_1} for more details on
	weighted Korobov spaces.
	
	We note that if the weights $\gamma_{\mathfrak{u}}$ are of the product
	form, i.e.
	\[
	\gamma_{\mathfrak{u}}=\prod_{j\in\mathfrak{u}}\gamma_{j}\qquad\text{for some positive }\text{\ensuremath{\gamma_{j}},\ensuremath{\quad\text{for }j=1,\dots,d,}}
	\]
	then the kernel $K_{\alpha}^{\mathrm{kor}}(\cdot,\cdot)$ can be written
	as the product of kernels:
	\begin{equation}
		K_{\alpha}^{\mathrm{kor}}(\rv{\boldsymbol{x},\boldsymbol{x}'})=\prod_{j=1}^{d}K_{1,\alpha,\gamma_{j}}^{\mathrm{kor}}(x_{j},x'_{j}),\label{eq:prod-Kor}
	\end{equation}
	with $K_{1,\alpha,\gamma_{j}}^{\mathrm{kor}}(x,x')=1+\gamma_{j}\sum_{h\in\mathbb{Z}\setminus\{0\}}\frac{e^{2\pi ih(x-x')}}{|h|^{\alpha}}$.
	
	Under this setting, the equation \eqref{eq:eq-disc} is equivalent
	to the following: Find $f_{M,N;\boldsymbol{Y}}^{\lambda}=\sum_{n=1}^{N}c_{n}K(\boldsymbol{x}_{n},\cdot)$
	such that
	\[
	\langle f_{M,N;\boldsymbol{Y}}^{\lambda},K_{\alpha}^{\mathrm{kor}}(\boldsymbol{x}_{k},\cdot)\rangle_{L^{2}([0,1]^{d})}+\lambda f_{M,N;\boldsymbol{Y}}^{\lambda}(\boldsymbol{x}_{k})=\frac{1}{M}\sum_{m=1}^{M}K_{\alpha}^{\mathrm{kor}}(\boldsymbol{x}_{k},Y_{m}(\omega)) 
	\]
	for $k=1,\dots,N$. 
	The linear system for $\boldsymbol{c}=(c_{1},\dots,c_{n})^{\top}$
	is given by \eqref{eq:eq-matrix-vector}, which can be written in
	a closed form for $\alpha$ even. Indeed, we have \[\langle K_{\alpha}^{\mathrm{kor}}(\boldsymbol{x}_{j},\cdot),K_{\alpha}^{\mathrm{kor}}(\boldsymbol{x}_{k},\cdot)\rangle_{L^{2}([0,1]^{d})}=\tilde{K}_{\alpha}^{\mathrm{kor}}(\boldsymbol{x}_{j},\boldsymbol{x}_{k})\]
	with
	\begin{equation}
		\tilde{K}_{\alpha}^{\mathrm{kor}}(\boldsymbol{x},\boldsymbol{x}'):=\sum_{\boldsymbol{h}\in\mathbb{Z}^{d}}r(\boldsymbol{h},\gamma)^{-2}e^{2\pi i\boldsymbol{h}\cdot(\boldsymbol{x}-\boldsymbol{x}')},\label{eq:def-lat-Ktilde}
	\end{equation}
	which can be written in a closed form with $B_{2\alpha}$.
	
	As the point set $\{\boldsymbol{x}_{k}\}_{k=1}^{N}$, we will consider
	the so-called \emph{rank-$1$ lattice} points. A rank-$1$ lattice
	point set $\{\boldsymbol{x}_{k}\}_{k=1}^{N}$ is given by
	\begin{equation}
		\boldsymbol{x}_{k}=\biggl\{\frac{k\boldsymbol{z}}{N}\biggr\}\quad\text{for }k=1,\dots,N,\label{eq:def-lat-pts}
	\end{equation}
	where $\boldsymbol{z}\in\{1,\dots,N\}^{d}$, and the braces around
	the vector of length $d$ indicate that each component of the vector
	is to be replaced by its fractional part. Because of the lattice structure
	of these points, the left-hand side of the equation \eqref{eq:eq-matrix-vector}
	but with $K_{\alpha}^{\mathrm{kor}}$ in place of $K$ becomes a circulant
	matrix, and thus the equation can be solved fast using the Fast Fourier
	Transform. See \cite[Section 2.2]{KaarniojaEtAl.V_2021_FastApproximationPeriodic}
	for an analogous argument.
	
	\begin{sloppypar}
		As stated in Proposition~\ref{prop:sq-bias}, the squared bias can
		be bounded by the $\mathcal{N}_{\mathrm{kor},\alpha}$-orthogonal
		projection error to the space spanned by $N$ functions $K_{\alpha}^{\mathrm{kor}}(\boldsymbol{x}_{k},\cdot)$\rv{,} $k=1,\dots,N$. 
		Under the setting of this section, the projection is given by the
		kernel interpolation. Notice that, given $N$, the integer vector
		$\boldsymbol{z}$ completely determines the lattice points \eqref{eq:def-lat-pts},
		and thus the corresponding kernel interpolant. In \cite{KaarniojaEtAl.V_2021_FastApproximationPeriodic},
		the present authors and co-authors obtained the following result on
		the choice of $\boldsymbol{z}$ and the resulting interpolation error,
		which we re-state here in our context.
	\end{sloppypar}
	\begin{proposition}
		\label{prop:interp-err-periodic}
		Given $d\geq1$, $\alpha>1$, weights
		$(\gamma_{\mathfrak{u}})_{\mathfrak{u}\subset\mathbb{N}}$ with $\gamma_{\emptyset}=1$,
		and prime $N$, a generating vector $\boldsymbol{z}$ can be constructed
		by a greedy algorithm called the component-by-component construction~\cite{Cools.R_etal_2020_LatticeAlgorithmsMultivariate,Cools.R_KNS_2021_FastCBC_MathComp}
		so that the $L^{2}$-approximation error of the kernel interpolant
		$\mathscr{I}_{N}f$ of the density $f\in\mathcal{N}_{\mathrm{kor},\alpha}$
		using $\boldsymbol{z}$ satisfies
		\[
		\|\mathscr{I}_{N}f-f\|_{L^{2}([0,1]^{d})}\leq C_{\alpha,\delta,d}\|f\|_{\mathrm{kor},\alpha}\frac{1}{N^{\alpha/4-\delta}}\quad\text{for every }\ \delta\in(0,\tfrac{\alpha}{4}),
		\]
		with  $C_{\alpha,\delta,d}:=\rv{C_{\alpha,\delta,d}\bigl((\gamma_{\mathfrak{u}})_{\mathfrak{u}\subset\mathbb{N}}\bigr):=}\big(\sum_{\mathfrak{\mathfrak{u}}\subset{\{1,\dots,d\}}}\max(|\mathfrak{u}|,1)\,\gamma_{\mathfrak{u}}^{\frac{1}{\alpha-4\delta}}\,[2\zeta\big(\tfrac{\alpha}{\alpha-4\delta}\big)]^{|\mathfrak{u}|}\big)^{\alpha-4\delta}$. 
		Here, $\zeta(x):=\sum_{k=1}^{\infty}k^{-x}$, $x>1$, denotes the
		Riemann zeta function. 
		The constant $C_{\alpha,\delta,d}$ depends on $\delta$ but can be bounded independently of $d$ provided that \rv{the weights satisfy} \begin{equation}\label{eq:periodic-thm-summability-cond}
			\sum_{\substack{\mathfrak{\mathfrak{u}}\subset\bbN\\
					\,|\mathfrak{u}|<\infty
				}
			}\max(|\mathfrak{u}|,1)\,\gamma_{\mathfrak{u}}^{\frac{1}{\alpha-4\delta}}\,[2\zeta\big(\tfrac{\alpha}{\alpha-4\delta}\big)]^{|\mathfrak{u}|}<\infty.
		\end{equation}
	\end{proposition}
	
	We now apply the variance estimate, Proposition~\ref{prop:var-estim}.
	Let $\mathcal{N}_{\mathrm{kor},\alpha}^{\tau}$ be the normed subspace
	of $L^{2}([0,1]^{d})$ defined by
	\[
	\mathcal{N}_{\mathrm{kor},\alpha}^{\tau}:=\{v\in L^{2}([0,1]^{d})\mid\|v\|_{\mathcal{N}_{\mathrm{kor},\alpha}^{\tau}}<\infty\},
	\]
	with $\|v\|_{\mathcal{N}_{\mathrm{kor},\alpha}^{\tau}}:=\big(\sum_{\boldsymbol{h}\in\mathbb{Z}^{d}}r(\boldsymbol{h},\gamma)^{\tau}|\langle v,e^{2\pi i\boldsymbol{h}\cdot\text{-}}\rangle_{L^{2}([0,1]^{d})}|^{2}\big)^{1/2}<\infty$,
	and let $\mathcal{N}_{\mathrm{kor},\alpha}^{-\tau}$ be the vector
	space of continuous linear functionals on $\mathcal{N}_{\mathrm{kor},\alpha}^{\tau}$.
	Note that, following an argument analogous to the proof of Proposition~\ref{prop:dual-isom},
	for $\Phi\in\mathcal{N}_{\mathrm{kor},\alpha}^{-\tau}=(\mathcal{N}_{\mathrm{kor},\alpha}^{\tau})'$
	we have
	\[
	\|\Phi\|_{\mathcal{N}_{\mathrm{kor},\alpha}^{-\tau}}=\Bigl(\sum_{\boldsymbol{h}\in\mathbb{Z}^{d}}r(\boldsymbol{h},\gamma)^{-\tau}|\Phi(e^{2\pi i\boldsymbol{h}\cdot\text{-}})|^{2}\Bigr)^{1/2}.
	\]
	To be able to invoke Proposition~\ref{prop:var-estim}, we first
	derive a lower bound for $\tau$ such that $\Delta_{\boldsymbol{Y}}$
	is in $\mathcal{N}_{\mathrm{kor},\alpha}^{-\tau}$.
	\begin{proposition}
		Let $\alpha>1$ and $\tau\in(0,1]$ be given. Then, the point evaluation
		functional $\delta_{\boldsymbol{x}}(v)=v(\boldsymbol{x})$ satisfies
		$\delta_{\boldsymbol{x}}\in\mathcal{N}_{\mathrm{kor},\alpha}^{-\tau}$
		for all $\boldsymbol{x}\in[0,1]^{d}$ if and only if $\tau>\frac{1}{\alpha}$.
	\end{proposition}
	
	\begin{proof}
		For $v\in\mathcal{N}_{\mathrm{kor},\alpha}^{\tau}$ we have
		\begin{align*}
			|\delta_{\boldsymbol{x}}(v)| & =\biggl|\sum_{\boldsymbol{h}\in\mathbb{Z}^{d}}\langle v,\mathrm{e}^{2\pi i\boldsymbol{h}\cdot\textrm{-}}\rangle_{L^{2}([0,1]^{d})}\mathrm{e}^{2\pi i\boldsymbol{h}\cdot\textrm{\ensuremath{\boldsymbol{x}}}}\biggr|\\
			& \leq\Bigl(\sum_{\boldsymbol{h}\in\mathbb{Z}^{d}}r(\boldsymbol{h},\gamma)^{-\tau}\Bigr)^{1/2}\Bigl(\sum_{\boldsymbol{h}\in\mathbb{Z}^{d}}r(\boldsymbol{h},\gamma)^{\tau}|\langle v,\mathrm{e}^{2\pi i\boldsymbol{h}\cdot\textrm{-}}\rangle_{L^{2}([0,1]^{d})}|^{2}\Bigr)^{1/2}\\
			& =\Bigl(\sum_{\boldsymbol{h}\in\mathbb{Z}^{d}}r(\boldsymbol{h},\gamma)^{-\tau}\Bigr)^{1/2}\|v\|_{\mathcal{N}_{\mathrm{kor},\alpha}^{\tau}},
		\end{align*}
		but for $\alpha\tau>1$ the first factor is bounded:
		\[
		\sum_{\boldsymbol{h}\in\mathbb{Z}^{d}}r(\boldsymbol{h},\gamma)^{-\tau}=1+\sum_{\emptyset\neq\mathfrak{u}\subseteq\{1,\ldots,d\}}\gamma_{\mathfrak{u}}^{\tau}\Biggl(\sum_{h\in\mathbb{Z}\setminus\{0\}}\frac{1}{|h|^{\alpha\tau}}\Biggr)^{|\mathfrak{u}|}<\infty,
		\]
		and thus if $\tau>\frac{1}{\alpha}$, then $\delta_{\boldsymbol{x}}\in\mathcal{N}_{\mathrm{kor},\alpha}^{-\tau}$
		for any $\boldsymbol{x}\in[0,1]^{d}$. Suppose now $\tau\leq\frac{1}{\alpha}$.
		If $\delta_{\boldsymbol{x}}$ is continuous on $\mathcal{N}_{\mathrm{kor},\alpha}^{\tau}$,
		then from Riesz representation theorem there exists $\Phi_{\boldsymbol{x}}\in\mathcal{N}_{\mathrm{kor},\alpha}^{\tau}$
		such that
		\[
		v(\boldsymbol{x})=\sum_{\boldsymbol{h}\in\mathbb{Z}^{d}}r(\boldsymbol{h},\gamma)^{\tau}\langle\Phi_{\boldsymbol{x}},e^{2\pi i\boldsymbol{h}\cdot\text{-}}\rangle_{L^{2}([0,1]^{d})}\overline{\langle v,e^{2\pi i\boldsymbol{h}\cdot\text{-}}\rangle_{L^{2}([0,1]^{d})}}\text{ for any }v\in\mathcal{N}_{\mathrm{kor},\alpha}^{\tau},
		\]
		where we let $\overline{\langle v,e^{2\pi i\boldsymbol{h}\cdot\text{-}}\rangle_{L^{2}([0,1]^{d})}}:=\int_{[0,1]^{d}}v(\bsy)e^{2\pi i\boldsymbol{h}\cdot\text{\ensuremath{\boldsymbol{y}}}}\mathrm{d}\boldsymbol{y}$.
		Choosing $v(\boldsymbol{x})=e^{2\pi i\tilde{\boldsymbol{h}}\cdot\boldsymbol{x}}$,
		$\tilde{\boldsymbol{h}}\in\mathbb{Z}^{d}$ yields $e^{2\pi i\tilde{\boldsymbol{h}}\cdot\boldsymbol{x}}=r(\tilde{\boldsymbol{h}},\gamma)^{\tau}\langle\Phi_{\boldsymbol{x}},e^{2\pi i\tilde{\boldsymbol{h}}\cdot\text{-}}\rangle_{L^{2}([0,1]^{d})}$
		, and thus $1=r(\tilde{\boldsymbol{h}},\gamma)^{2\tau}|\langle\Phi_{\boldsymbol{x}},e^{2\pi i\tilde{\boldsymbol{h}}\cdot\text{-}}\rangle_{L^{2}([0,1]^{d})}|^{2}$.
		Hence, we obtain
		\[
		\sum_{\boldsymbol{h}\in\mathbb{Z}^{d}}r(\boldsymbol{h},\gamma)^{-\tau}=\sum_{\boldsymbol{h}\in\mathbb{Z}^{d}}r(\boldsymbol{h},\gamma)^{\tau}|\langle\Phi_{\boldsymbol{x}},e^{2\pi i\boldsymbol{h}\cdot\text{-}}\rangle_{L^{2}([0,1]^{d})}|^{2}.
		\]
		From $\Phi_{\boldsymbol{x}}\in\mathcal{N}_{\mathrm{kor},\alpha}^{\tau}$,
		the right hand side is convergent, whereas for $\tau\alpha\leq1$
		the left hand side is divergent, a contradiction. Hence, $\delta_{\boldsymbol{x}}$
		is not continuous on $\mathcal{N}_{\mathrm{kor},\alpha}^{\tau}$.
	\end{proof}
	Indeed, for $\alpha\tau>1$ the space $\mathcal{N}_{\mathrm{kor},\alpha}^{\tau}$
	is a reproducing kernel Hilbert space. To see this, let
	\begin{align*}
		K_{\alpha;\tau}^{\mathrm{kor}}(\boldsymbol{x},\boldsymbol{x}') & :=\sum_{\boldsymbol{h}\in\mathbb{Z}^{d}}r(\boldsymbol{h},\gamma)^{-\tau}e^{2\pi i\boldsymbol{h}\cdot(\boldsymbol{x}-\boldsymbol{x}')}\\
		& =1+\sum_{\emptyset\neq\mathfrak{u}\subseteq\{1,\ldots,d\}}\gamma_{\mathfrak{u}}^{\tau}K_{\alpha;\tau}(\boldsymbol{x}_{\mathfrak{u}},\boldsymbol{x}'_{\mathfrak{u}}),
	\end{align*}
	where
	\[
	K_{\alpha;\tau}^{\mathrm{kor}}(\boldsymbol{x}_{\mathfrak{u}},\boldsymbol{x}'_{\mathfrak{u}})=\prod_{j\in\mathfrak{u}}\Biggl(\sum_{h\in\mathbb{Z}\setminus\{0\}}\frac{e^{2\pi ih(x_{j}-x'_{j})}}{|h|^{\alpha\tau}}\Biggr).
	\]
	If $\alpha\tau>1$, the series is uniformly absolutely-convergent
	and thus the kernel is continuous on $[0,1]^{d}$; if $\alpha\tau\leq1$,
	then the series $K_{\alpha;\tau}(\boldsymbol{0},\boldsymbol{0})$
	is divergent. Hence, $K_{\alpha;\tau}^{\mathrm{kor}}(\boldsymbol{x},\boldsymbol{x}')$
	is a reproducing kernel if and only if $\alpha\tau>1$.
	
	We conclude this section with the following estimates.
	\begin{theorem}
		\label{thm:periodic-thm}Let $\alpha>1$. Fix $\delta\in(0,\tfrac{\alpha}{4})$
		and $\tau\in(\tfrac{1}{\alpha},1]$ arbitrarily.
		\rv{Let the weights $(\gamma_{\mathfrak{u}})_{\mathfrak{u}\subset\mathbb{N}}$ satisfy $\gamma_{\emptyset}=1$.}  
		Then, for $f\in\mathcal{N}_{\mathrm{kor},\alpha}$ we have
		\begin{align*}
			\bbE\Bigl[\int_{\rv{[0,1]^d}}&|f_{M,N;\boldsymbol{Y}}^{\lambda}(\rv{\bs{x}})-f(\rv{\bs{x}})|^{2}\mathrm{d}\rv{\bs{x}}\Bigr]\\
			&\leq C_{\rv{\alpha,\delta,\tau,d}}\biggl(\|f\|_{\mathrm{kor},\alpha}^{2}\frac{1}{N^{\alpha/2-2\delta}}+\lambda\|f\|_{\mathrm{kor},\alpha}^{2}+\frac{\|f\|_{L^{2}([0,1]^{d})}}{M\lambda^{\tau}}\biggr),
		\end{align*}
		where the constant $\rv{C_{\alpha,\delta,\tau,d}=C_{\alpha,\delta,
				\tau,d}\bigl((\gamma_{\mathfrak{u}})_{\mathfrak{u}\subset\mathbb{N}}\bigr)}>0$ depends on $\alpha$,
		$\delta$, $\tau$ but 
		\rv{can be bounded independently of $d$ provided that \eqref{eq:periodic-thm-summability-cond} holds.}
		Moreover,
		if $f$ is in the RKHS $(\mathcal{N}_{\widetilde{\mathrm{kor}},\alpha},\|\cdot\|_{\widetilde{\mathrm{kor}},\alpha})$
		associated with the kernel \eqref{eq:def-lat-Ktilde}, then the bound
		improves to
		\begin{align*}
			\bbE\Bigl[\int_\rv{[0,1]^d}|&f_{M,N;\boldsymbol{Y}}^{\lambda}(\rv{\bs{x}})-f(\rv{\bs{x}})|^{2}\mathrm{d}\rv{\bs{x}}\Bigr]\\
			&\leq
			\widetilde{C}_{\rv{\alpha,\delta,\tau,d}}\biggl(\|f\|_{\mathrm{kor},\alpha}^{2}\frac{1}{N^{\alpha/2-2\delta}}+\lambda^{2}\|f\|_{\mathrm{\widetilde{\mathrm{kor}}},\alpha}^{2}+\frac{\|f\|_{L^{2}([0,1]^{d})}}{M\lambda^{\tau}}\biggr),
		\end{align*}
		\rv{where the constant $\rv{\widetilde{C}_{\alpha,\delta,\tau,d}=\widetilde{C}_{\alpha,\delta,
					\tau,d}\bigl((\gamma_{\mathfrak{u}})_{\mathfrak{u}\subset\mathbb{N}}\bigr)}>0$ depends on $\alpha$,
			$\delta$, $\tau$ but 
			can be bounded independently of $d$ provided that the weights corresponding to the kernel \eqref{eq:def-lat-Ktilde} satisfy the summability condition analogous to \eqref{eq:periodic-thm-summability-cond}.}
	\end{theorem}
	\begin{proof}
		From \eqref{eq:general-result} in Theorem~\ref{thm:general-summary},
		Proposition~\ref{prop:interp-err-periodic} implies
		\begin{align*}
			\bbE\Bigl[\int_{[0,1]^{d}}&|f_{M,N;\boldsymbol{Y}}^{\lambda}(\rv{\bs{x}})-f(\rv{\bs{x}})|^{2}\mathrm{d}\rv{\bs{x}}\Bigr]\\
			&\leq C_{\alpha,\delta}^{2}\|f\|_{\mathrm{kor},\alpha}^{2}\frac{1}{N^{\alpha/2-2\delta}}+\lambda\|f\|_{\mathrm{kor},\alpha}^{2}+\frac{\langle K_{\alpha;\tau}^{\mathrm{kor}}(\cdot,\cdot),f\rangle_{L^{2}(\rv{[0,1]^d})}}{M\lambda^{\tau}},
		\end{align*}
		for any $\delta\in(0,\frac{\alpha}{4})$ and $\tau\in(\tfrac{1}{\alpha},1]$,
		where we note that
		\[
		\langle K_{\alpha;\tau}^{\mathrm{kor}}(\cdot,\cdot),f\rangle_{L^{2}([0,1]^{d})}\leq\|f\|_{L^{2}([0,1]^{d})}\sum_{\boldsymbol{h}\in\mathbb{Z}^{d}}r(\boldsymbol{h},\gamma)^{-\tau}<\infty.
		\]
		The proof for the second claim follows from an analogous argument
		using \eqref{eq:general-result-2}.
	\end{proof}
	Given that the density is sufficiently smooth, we obtain the MISE
	convergence rate arbitrarily close to $M^{-1/(1+\frac{1}{\alpha})}$,
	independently of the dimension $d$.
	\begin{corollary}
		\label{cor:final}Let $\alpha>1$. Fix $\delta\in(0,\tfrac{\alpha}{4})$
		and $\epsilon\in(0,1-\frac{1}{\alpha}]$ arbitrarily. 
		\rv{Suppose that weights
			$(\gamma_{\mathfrak{u}})_{\mathfrak{u}\subset\mathbb{N}}$ satisfy $\gamma_{\emptyset}=1$.} 
		Then, \rv{for $f\in\mathcal{N}_{\mathrm{kor},\alpha}$,} choosing
		$\lambda^{*}=M^{-1/(1+\frac{1}{\alpha}+\epsilon)}$, $N^{*}=\mathcal{O}(M^{1/(\frac{\alpha}{2}(1+\epsilon)+\frac{1}{2}-\frac{\delta}{2}(1+\frac{1}{\alpha}+\epsilon))})$,
		and $\tau^{*}=\frac{1}{\alpha}+\epsilon$ in Theorem~\ref{thm:periodic-thm}
		yields
		\[
		\bbE\Bigl[\int_\rv{[0,1]^d}|f_{M,N;\boldsymbol{Y}}^{\lambda}(\rv{\bs{x}})-f(\rv{\bs{x}})|^{2}\mathrm{d}\rv{\bs{x}}\Bigr]\leq C_{\alpha,\delta\rv{,d}}\|f\|_{\mathrm{kor},\alpha}^{2}M^{-1/(1+\frac{1}{\alpha}+\epsilon)}.
		\]
		Moreover, if $f$ is in the RKHS $(\mathcal{N}_{\widetilde{\mathrm{kor}},\alpha},\|\cdot\|_{\widetilde{\mathrm{kor}},\alpha})$
		associated with the kernel \eqref{eq:def-lat-Ktilde}, then by choosing     $\tau^{*}=\frac{1}{\alpha}+2\epsilon$,
		$\lambda^{*}=M^{-1/(2+\frac{1}{\alpha}+2\epsilon)}$,  and $N^{*}=\mathcal{O}(M^{1/(\frac{\alpha}{2}(1+\epsilon)+1-\frac{\delta}{2}(1+\frac{1}{2\alpha}+\epsilon))})$
		in Theorem~\ref{thm:periodic-thm}, we have a rate asymptotically
		faster in $\alpha$
		\[
		\bbE\Bigl[\int_\rv{[0,1]^d}|f_{M,N;\boldsymbol{Y}}^{\lambda}(\rv{\bs{x}})-f(\rv{\bs{x}})|^{2}\mathrm{d}\rv{\bs{x}}\Bigr]\leq C_{\alpha,\delta\rv{,d}}\|f\|_{\widetilde{\mathrm{kor}},\alpha}^{2}M^{-1/(1+\frac{1}{2\alpha}+\epsilon)}.
		\]
		\rv{The constant $C_{\alpha,\delta,d}=C_{\alpha,\delta,d}\bigl((\gamma_{\mathfrak{u}})_{\mathfrak{u}\subset\mathbb{N}}\bigr)>0$ can be bounded independently of $d$ provided that \eqref{eq:periodic-thm-summability-cond} holds.}
	\end{corollary}
	\rv{
		It turns out that the rate established above is almost minimax when $\alpha$ is an even integer. 
		For simplicity, let us consider the equal weights  $\gamma_{\mathfrak{u}}=1$ for all $\mathfrak{u}\subset\{1,\dots,d\}$.  
		For general weights, the following inequality still holds true, up to a constant depending on $(\gamma_{\mathfrak{u}})_{\mathfrak{u}\subset\{1,\dots,d\}}$. 
		
		Indeed, from the classical asymptotic minimax rate for the Korobov spaces  $\mathcal{N}_{\mathrm{kor},\alpha,1}$ for $d=1$ and $\alpha$ even, 
		(see \cite[Example 1]{MR711898}, also \cite{Efromovich.S_2010_OrthogonalSeriesDensity}, where we note that the definition of $\alpha$ in \cite{MR711898} is different from ours by a factor of $2$), for $M$ sufficiently large we have 
		\begin{align*}
			C_\alpha M^{-1/(1+\frac1{\alpha})}
			&\leq
			\inf_{\hat{f}}
			\sup_{\substack{f\in\mathcal{N}_{\mathrm{kor},\alpha,1} \\ \|f\|_{\mathrm{kor},\alpha,1}\leq1}}
			\bbE\Bigl[\int_\rv{[0,1]}|\hat{f}(x)-f(x)|^{2}\Bigr]\mathrm{d}x\\
			&\leq
			\inf_{\hat{f}}
			\sup_{\substack{f\in\mathcal{N}_{\mathrm{kor},\alpha} \\ \|f\|_{\mathrm{kor},\alpha}\leq1}}
			\bbE\Bigl[\int_\rv{[0,1]^d}|\hat{f}(\bs{x})-f(\bs{x})|^{2}\Bigr]\mathrm{d}\bs{x}\\
			&\leq
			\sup_{\substack{f\in\mathcal{N}_{\mathrm{kor},\alpha} \\ \|f\|_{\mathrm{kor},\alpha}\leq1}}
			\bbE\Bigl[\int_\rv{[0,1]^d}|f_{M,N;\boldsymbol{Y}}^{\lambda}(\bs{x})-f(\bs{x})|^{2}\Bigr]\mathrm{d}\bs{x}\leq C_{\alpha,\delta,d} M^{-1/(1+\frac1{\alpha}+\epsilon)},
		\end{align*}
		where the infimum is taken over all possible estimates $\hat{f}=\hat{f}(Y_1,\dots,Y_M)$. 
		Here, 
		the first inequality follows from \cite[Example 1]{MR711898},   
		in the second inequality we used that $f\in\mathcal{N}_{\mathrm{kor},\alpha,1}$ can be seen as a function $\tilde{f}\in\mathcal{N}_{\mathrm{kor},\alpha}$ on $[0,1]^d$ depending only on one variable with $\|f\|_{\mathrm{kor},\alpha,1}=\|\tilde{f}\|_{\mathrm{kor},\alpha}$ for $\alpha /2\in\mathbb{N}$ with equal weights,  
		and the last inequality is from Corollary~\ref{cor:final}. 
		Hence, we conclude that the rate $\mathcal{O}(M^{-1/(1+\frac{1}{\alpha}+\epsilon)})$ as in Corollary~\ref{cor:final} is asymptotically minimax up to $\epsilon>0$ for $f\in\mathcal{N}_{\mathrm{kor},\alpha}$. 
		Similarly, for $f\in\mathcal{N}_{\widetilde{\mathrm{kor}},\alpha}$ the rate $\mathcal{O}(M^{-1/(1+\frac{1}{2\alpha})})$ is asymptotically minimax, at least for $\alpha$ integer.
	}
	\section{Numerical results\label{sec:numerical}}
	
	We consider the density function
	\begin{equation}
		f(\boldsymbol{y})=\prod_{j=1}^{d}\biggl(1+\frac{1}{j^{4}}B_{4}(y_{j})\biggr),\quad\boldsymbol{y}\in[0,1]^{d},\label{eq:NR-exact}
	\end{equation}
	with respect to the uniform measure. Notice that, from $B_{4}>-1$
	on $[0,1]$ and $\int_{0}^{1}B_{4}(t)\mathrm{d}t=0$ the function
	$f$ is indeed a density function. The dimensions we consider are
	$d=6$ and $d=15$. 
	
	\sloppy{We generate samples $Y_{m}\sim\int_{\cdot}f(\boldsymbol{y})\mathrm{d}\boldsymbol{y}$
		using the Acceptance-Rejection method\rv{; see for example \cite[Section 2.2]{Glasserman.P_2003_Book_MCinFinancialEngineering}.}
		The generating vector for the
		lattice points are obtained by the component-by-component algorithm
		in \cite{Cools.R_etal_2020_LatticeAlgorithmsMultivariate}. The experiments
		were implemented in Julia 1.6.0 \cite{BezansonEtAl.J_2017_JuliaFreshApproach}.
		We approximate the MISE by}
	\begin{align*}
		\bbE\biggl[&\int_\rv{[0,1]^d}|f_{M,N;\boldsymbol{Y}}^{\lambda}(\rv{\bs{x}})-f(\rv{\bs{x}})|^{2}\mathrm{d}\rv{\bs{x}}\biggr]\\
		&\approx\frac{1}{S}\sum_{k=1}^{S}\frac{1}{100N}\sum_{n=1}^{N}\sum_{\ell=1}^{100}|f_{M,N;\boldsymbol{Y}^{(k)}}^{\lambda}(\{\boldsymbol{x}_{n}+\boldsymbol{p}_{\ell}\})-f(\{\boldsymbol{x}_{n}+\boldsymbol{p}_{\ell}\})|^{2},
	\end{align*}
	where $\boldsymbol{x}_{n}$, $n=1,\dots,N$ \rv{are} the lattice points
	used to determine the approximation space, $\boldsymbol{p}_{\ell}$,
	$\ell=1,\dots,100$ are a set of Sobol' points generated by the Julia
	function \texttt{SobolSeq}~\cite{Johnson.S_SobolModuleJulia},
	and the braces around the vector of length $d$ indicate that each
	component of the vector is to be replaced by its fractional part.
	This choice of evaluation points admits a fast evaluation; see \cite{KaarniojaEtAl.V_2021_FastApproximationPeriodic}
	for more details. The number $S$ of Monte Carlo replications $\boldsymbol{Y}^{(1)},\dots,\boldsymbol{Y}^{(S)}$
	is chosen such that the estimated confidence interval is smaller than
	the estimated MISE at least by a factor of $10$. As a kernel, we
	consider the Korobov kernel with product weights $\gamma_{\mathfrak{u}}=\prod_{j\in\mathfrak{u}}\frac{1}{j^{\alpha}}$
	as in \eqref{eq:prod-Kor}, with \rv{$\alpha\in \{2,4\}$}.  
	\rv{Observe that $f$ as in \eqref{eq:NR-exact} satisfies $f\in\mathcal{N}_{\mathrm{kor},\alpha}$ with $\alpha\in  (1,4]$.}
	
	Fig.~\ref{fig:d6lambig} shows a decay of MISE for $N=5,7,11$, with 
	\rv{various values of $\lambda$ between $0.8$ and $0.01$ }
	for $d=6$. Firstly, we report that increasing $N$ beyond $N=11$
	did not help decreasing the error. We interpret this as the projection
	error in Theorem~\ref{thm:periodic-thm} being negligible, and we
	focus, therefore, on the other two error terms in Theorem~\ref{thm:periodic-thm}.
	For the sample size $M$ large ($M=10^{5},10^{6}$), we observe that
	the error decreases as $\lambda$ becomes smaller. This supports Theorem~\ref{thm:periodic-thm},
	where the term $\mathcal{O}(1/(M\lambda^{\tau}))$, $\tau\in(1/\alpha,1)$
	with large $M$ would be negligible, and the term $\mathcal{O}(\lambda)$
	would be dominant. For smaller values of $M$, from Theorem~\ref{thm:periodic-thm}
	we expect that the term $\mathcal{O}(1/(M\lambda^{\tau}))$ is dominant
	relative to $\mathcal{O}(\lambda)$, and that as $M$ increases the
	MISE decays. This is precisely what we observe in Fig.~\ref{fig:d6lambig}.
	We also see that the values of $\alpha$ affect the decay in $M$
	only up to a constant, which again indicates the validity of our theory.
	We report a similar behaviour of the error for $d=15$ in Fig.~\ref{fig:d15lambig},
	which further supports our theory.
	\begin{figure}
		\includegraphics[width=0.99\textwidth]{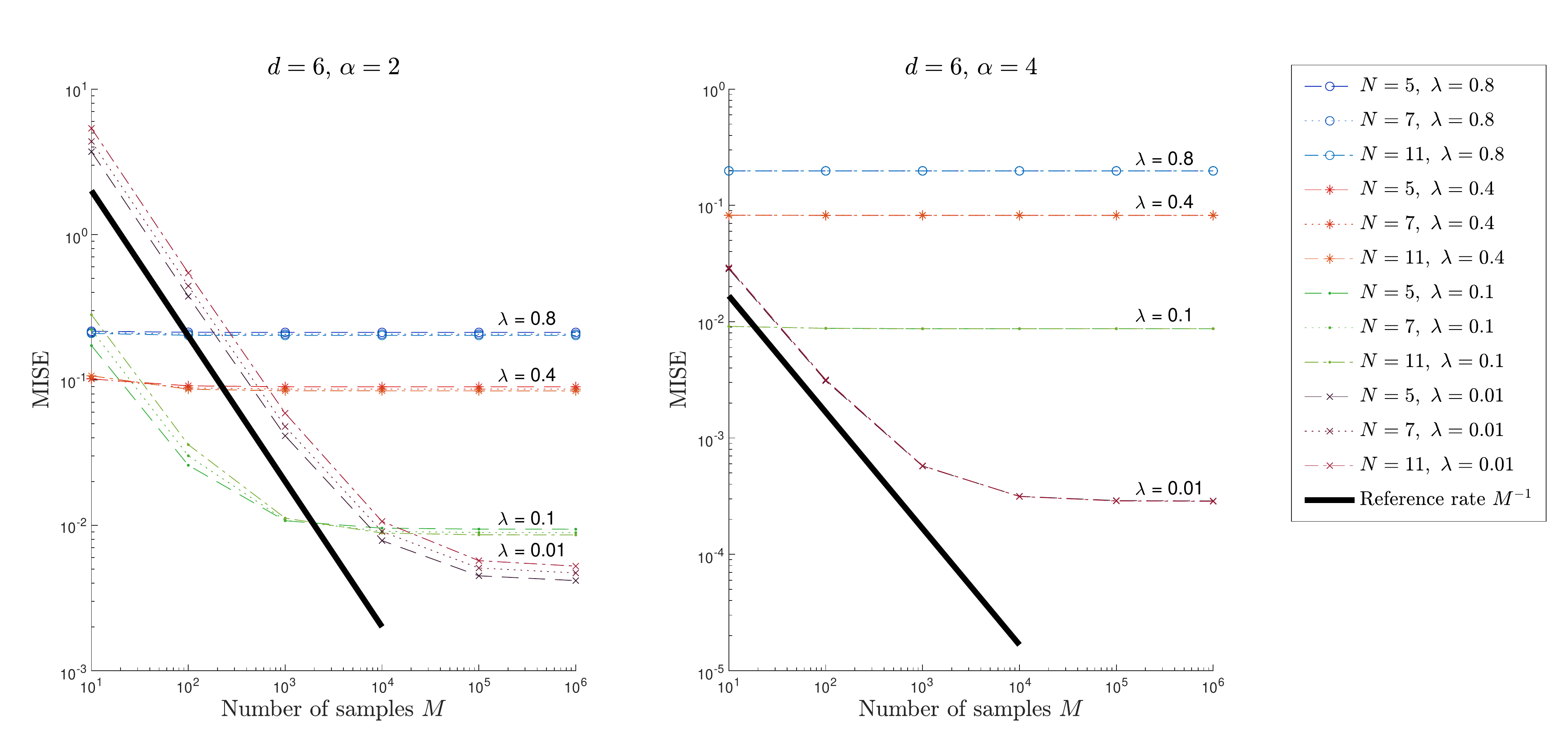}
		
		\caption{Plot of MISE with $d=6$, $\alpha=2\text{ (left)},4\text{ (right)}$
			, and $\lambda=0.8,\dots,0.01$ with varying sample size~$M$}
		
		\label{fig:d6lambig}
	\end{figure}
	\begin{figure}
		\includegraphics[width=0.99\textwidth]{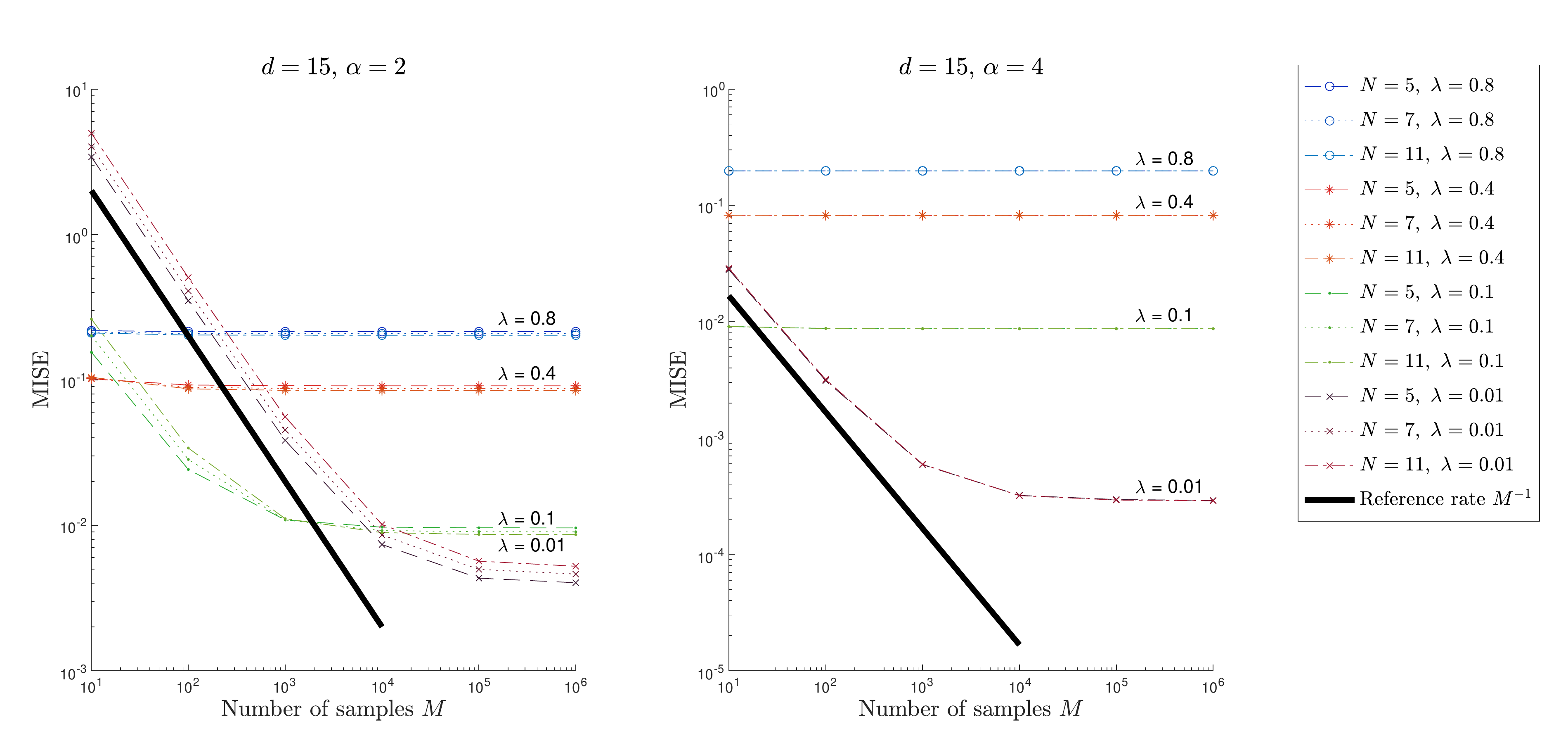}
		
		\caption{Plot of MISE with $d=15$, $\alpha=2\text{ (left)},4\text{ (right)}$,
			and $\lambda=0.8,\dots,0.01$ with varying sample size~$M$}
		
		\label{fig:d15lambig}
	\end{figure}
	
	Next, we will see the behaviour of MISE for even smaller values of
	$\lambda$. Figs.~\ref{fig:d6lamsmall} and~\ref{fig:d15lamsmall}
	show a decay of MISE for 
	\rv{various values of $\lambda$ between $0.1$ and $0.0001$}
	with $N=11$
	and varying sample size $M$. Like in the previous case, we observe
	that the MISE decays in $M$ with rate $M^{-1}$ until it reaches
	a plateau. However, unlike in the previous case, we see that even
	on the plateau the error may be larger for smaller $\lambda$. To
	understand this observation, we next see the behaviour of the MISE
	in $\lambda$, with other parameters being fixed.
	\begin{figure}
		\includegraphics[width=0.99\textwidth]{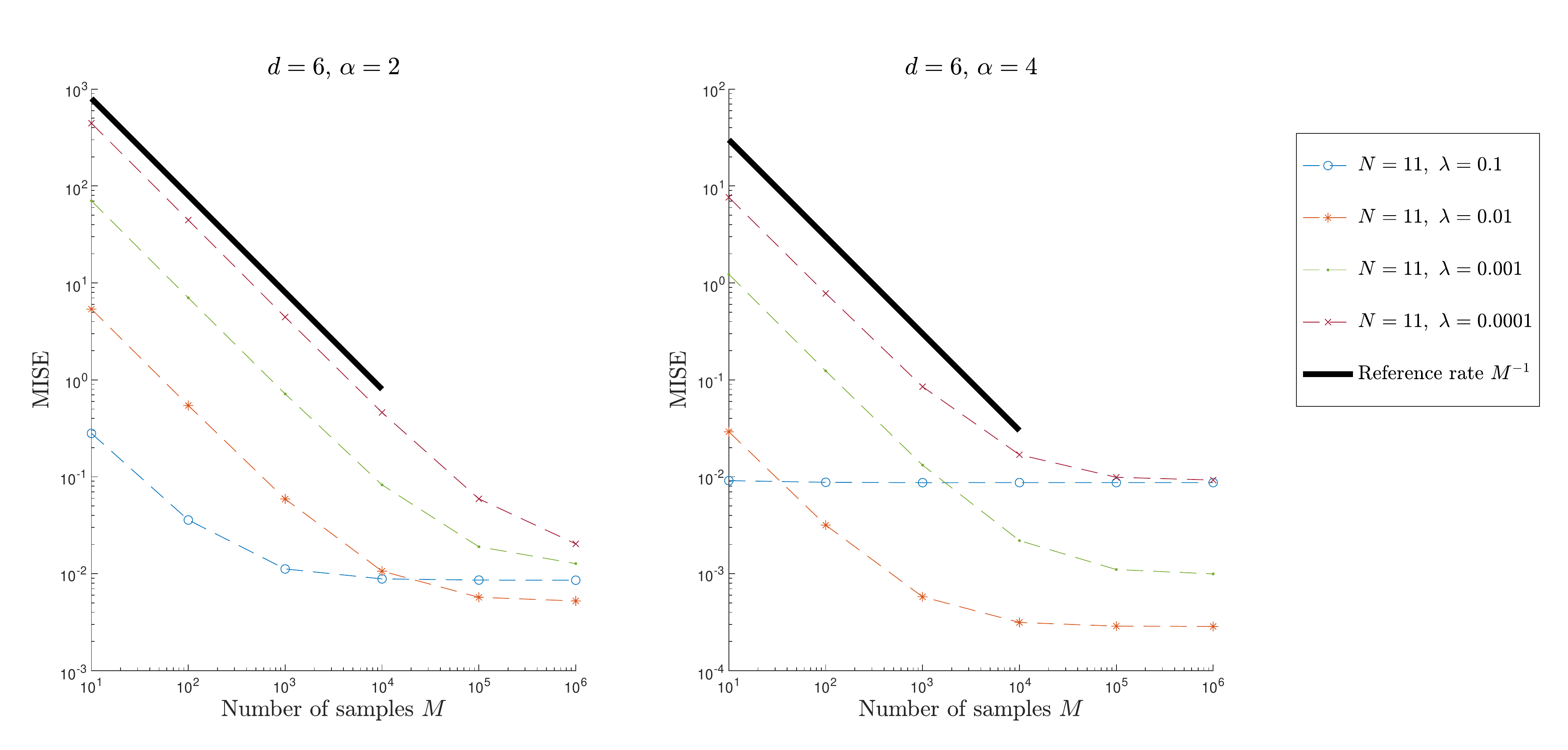}
		
		\caption{Plot of MISE for $d=6$, $\alpha=2\text{ (left)},4\text{ (right)}$,
			and $\lambda=0.1,\dots,0.0001$ with varying sample size~$M$.}
		
		\label{fig:d6lamsmall}
	\end{figure}
	
	\begin{figure}
		\includegraphics[width=0.99\textwidth]{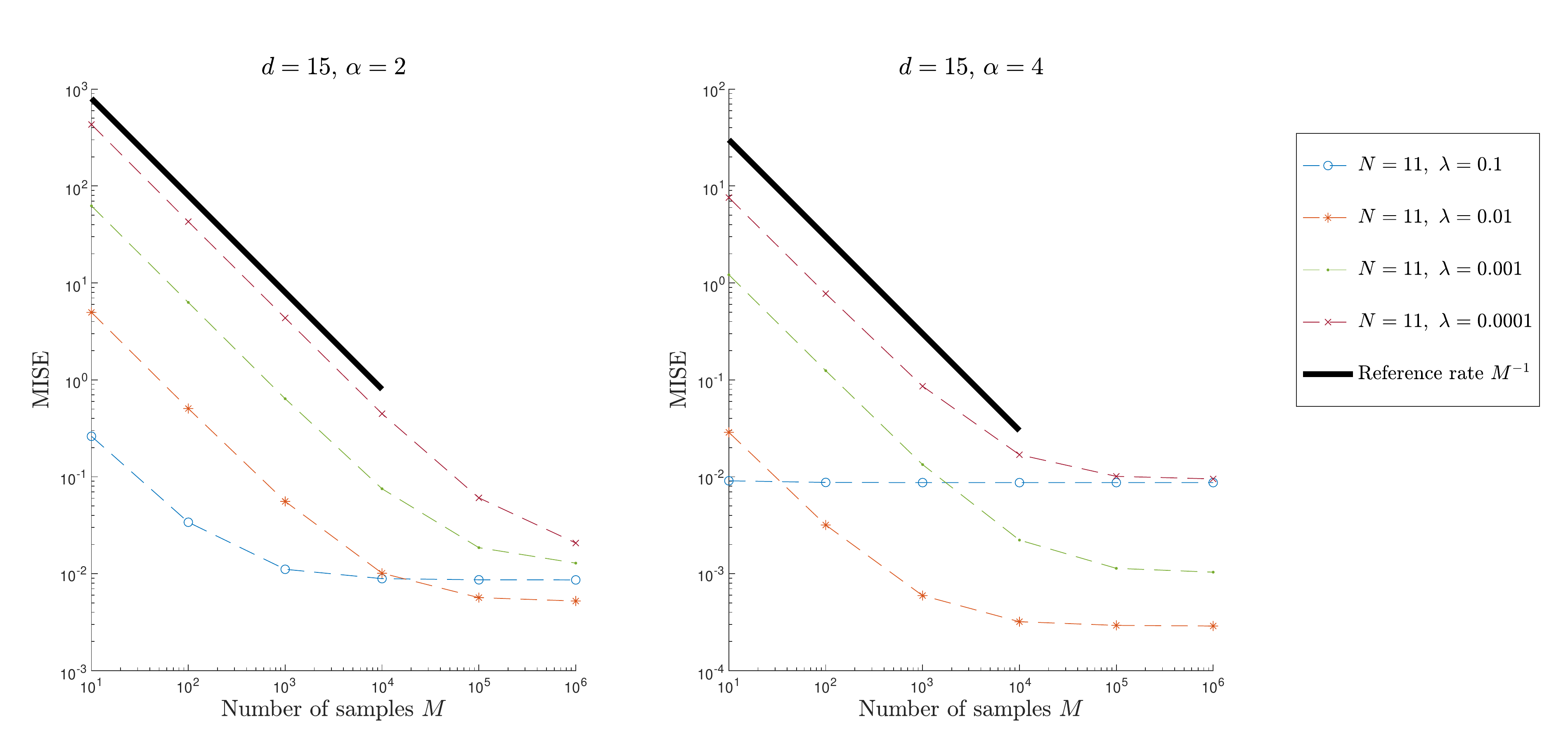}
		
		\caption{Plot of MISE for $d=15$, $\alpha=2\text{ (left)},4\text{ (right)}$,
			and $\lambda=0.1,\dots,0.0001$ with varying sample size~$M$.}
		
		\label{fig:d15lamsmall}
	\end{figure}
	
	Figs.~\ref{fig:changelam-d6} and \ref{fig:changelam-d15} show
	MISE with $N=11$ and $M=10^{4}$, with $\lambda=0.7^{k}$, $k=0,\dots,70$
	for $d=6$ and $15$.
	
	For $d=6$ and $\alpha=2$, from $\lambda=1$ to $\lambda=0.7^{9}\approx0.04$,
	we see the MISE decays almost quadratically in $\lambda$; see the
	graphs on the left in Fig.~\ref{fig:changelam-d6}. Then, the MISE
	increases as $\lambda$ becomes small until it reaches a plateau.
	Although the rate of this increase in the regime of $\lambda=0.7^{10}\approx0.03,\dots,0.7^{25}\approx0.001$
	is faster than the rate $\lambda^{-1/\alpha}$ that is anticipated
	by Theorem~\ref{thm:periodic-thm}, asymptotically the observed rate
	of increase is better than this theoretical rate. This observation
	may suggest that our theoretical rate is not sharp, but it is consistent
	with our theory. For $\alpha=4$ we observe a similar behaviour in
	the graphs on the right in Fig.~\ref{fig:changelam-d6}: the MISE
	decays, and then increases until it reaches a plateau. This observation
	again supports Theorem~\ref{thm:periodic-thm}. For $d=15$, we observe
	similar results, which further supports our theory; see Fig.~\ref{fig:changelam-d15}.
	
	\begin{figure}
		\includegraphics[width=0.99\textwidth]{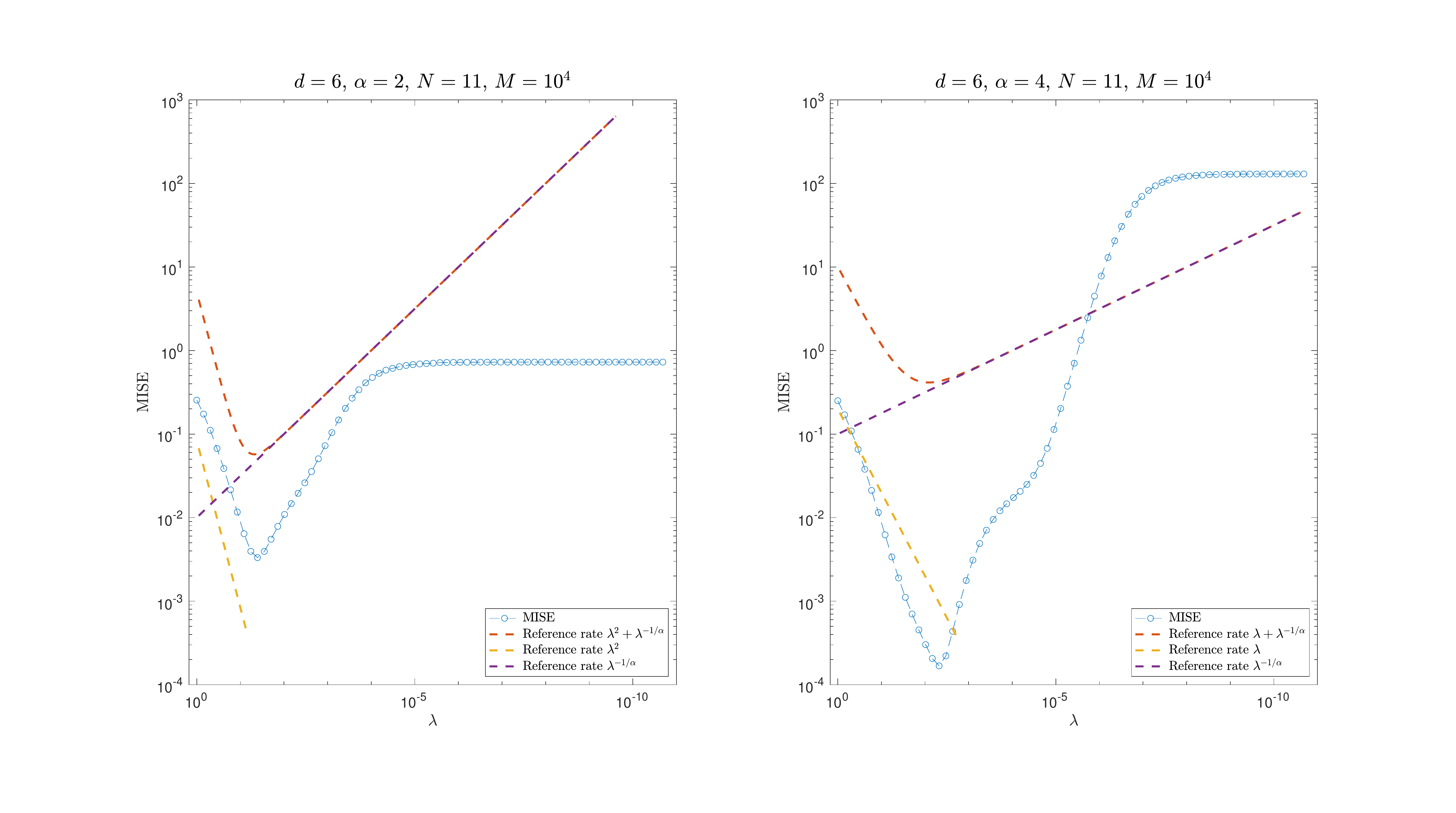}
		
		\caption{Plot of MISE for $d=6$, $\alpha=2\text{ (left)},4\text{ (right)}$,
			$N=11$, and $M=10^{4}$ with $\lambda=0.7^{k},k=0,\dots,70$.}
		
		\label{fig:changelam-d6}
	\end{figure}
	
	\begin{figure}
		\includegraphics[width=0.99\textwidth]{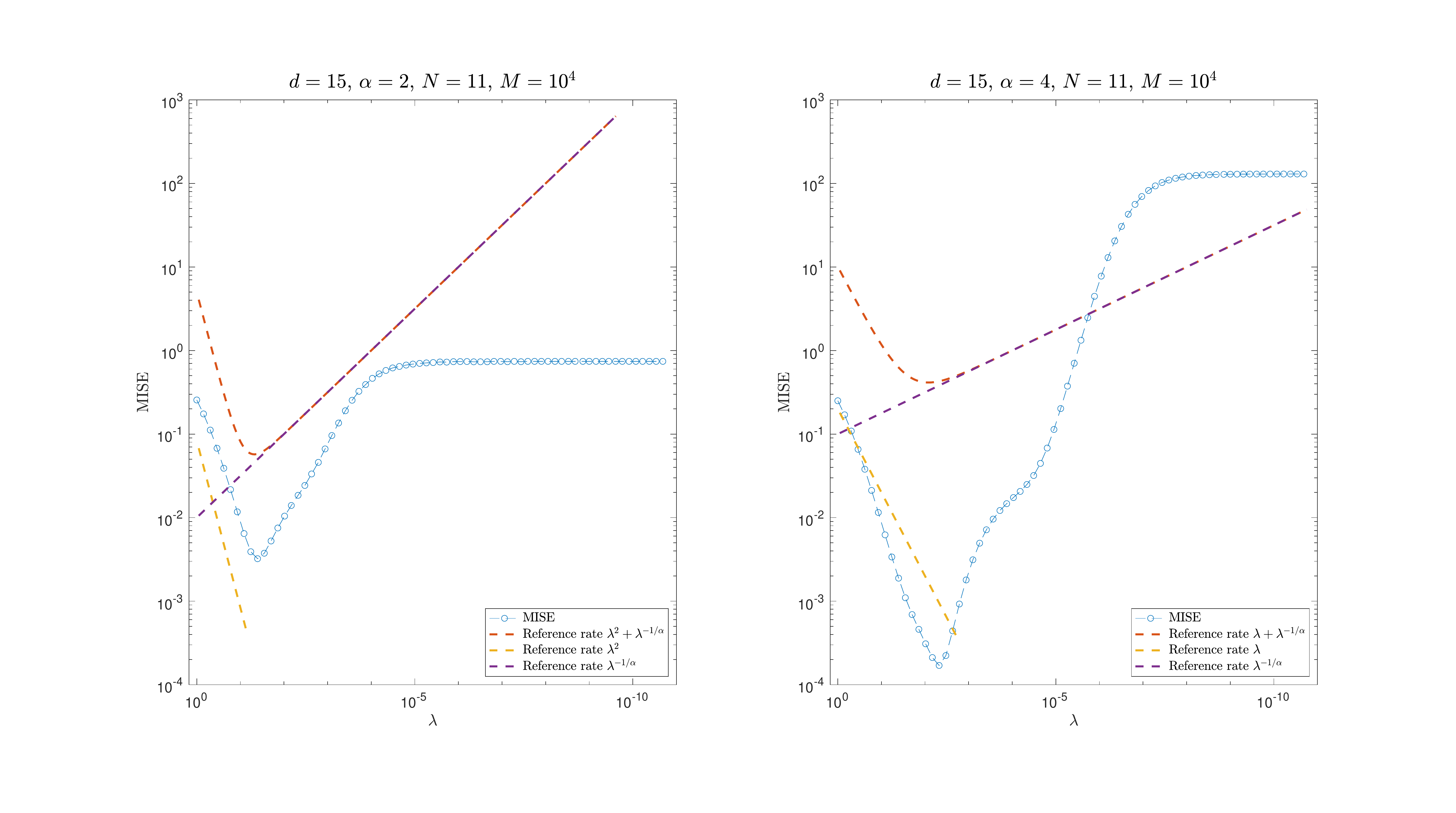}
		
		\caption{Plot of MISE for $d=15$, $\alpha=2\text{ (left)},4\text{ (right)}$,
			$N=11$, and $M=10^{4}$ with $\lambda=0.7^{k},k=0,\dots,70$.}
		
		\label{fig:changelam-d15}
	\end{figure}
	
	Finally, we show in Figs.~\ref{fig:changelam(M)-d6} and \ref{fig:changelam(M)-d15}
	the error decay for $\lambda=\lambda(M,\alpha)$ depending on $M$
	and $\alpha$ as in Corollary~\ref{cor:final}. For $\alpha=2$,
	we let $\lambda=1000M^{-\frac{1}{1+1/\alpha}}$, and for $\alpha=4$
	we let $\lambda=5000M^{-\frac{1}{1+1/\alpha}}$, with $M=10^{k}$,
	$k=3,\dots,7$. In both Figs., we observe that the MISE decays with
	rate essentially $M^{-\frac{1}{1+1/\alpha}}$, which supports our
	theory.
	\begin{figure}
		\includegraphics[width=0.99\textwidth]{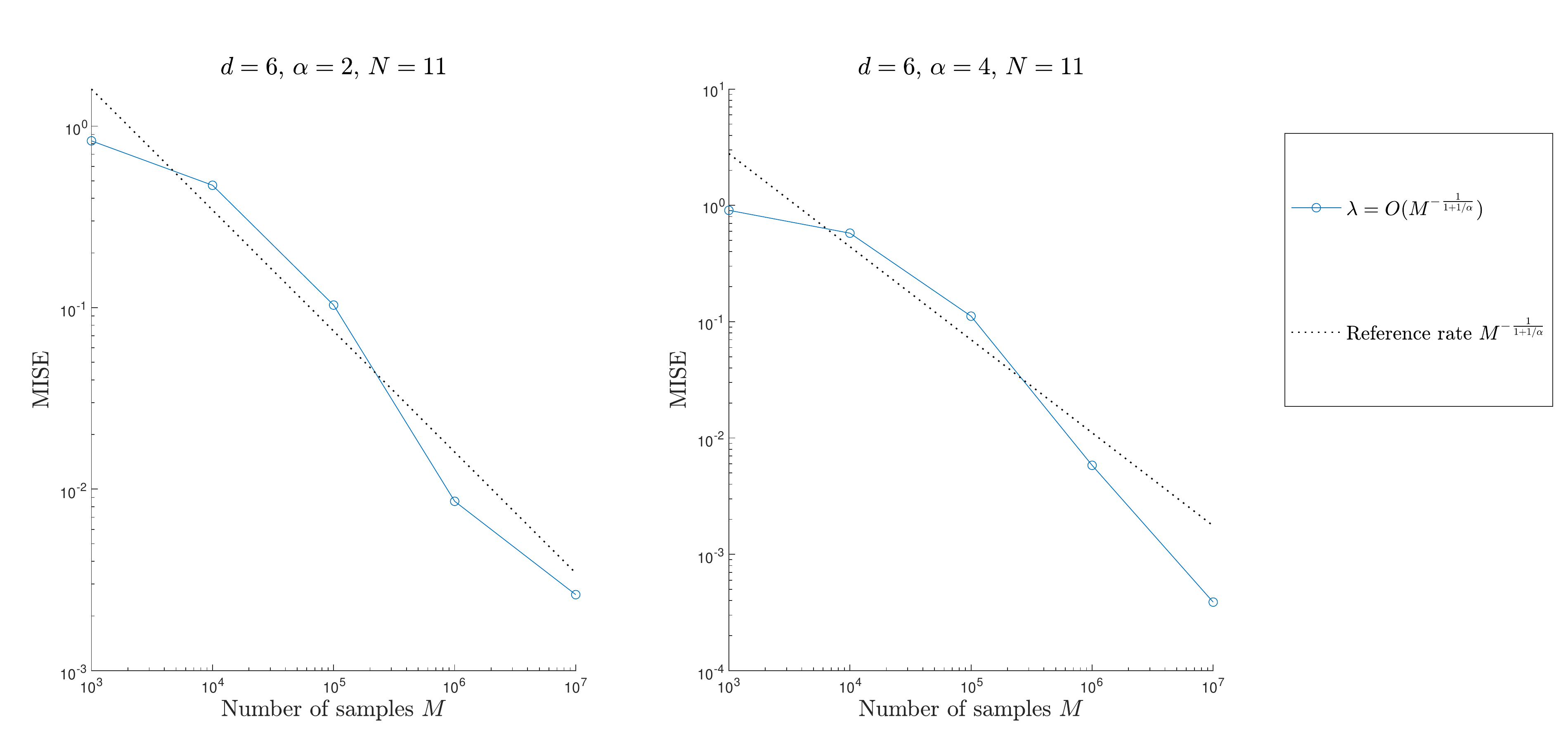}
		
		\caption{Plot of MISE for $d=6$ with $N=11$, $\alpha=2\text{ (left)},4\text{ (right)}$,
			and $\lambda=O(M^{-\frac{1}{1+1/\alpha}})$.}
		
		\label{fig:changelam(M)-d6}
	\end{figure}
	
	\begin{figure}
		\includegraphics[width=0.99\textwidth]{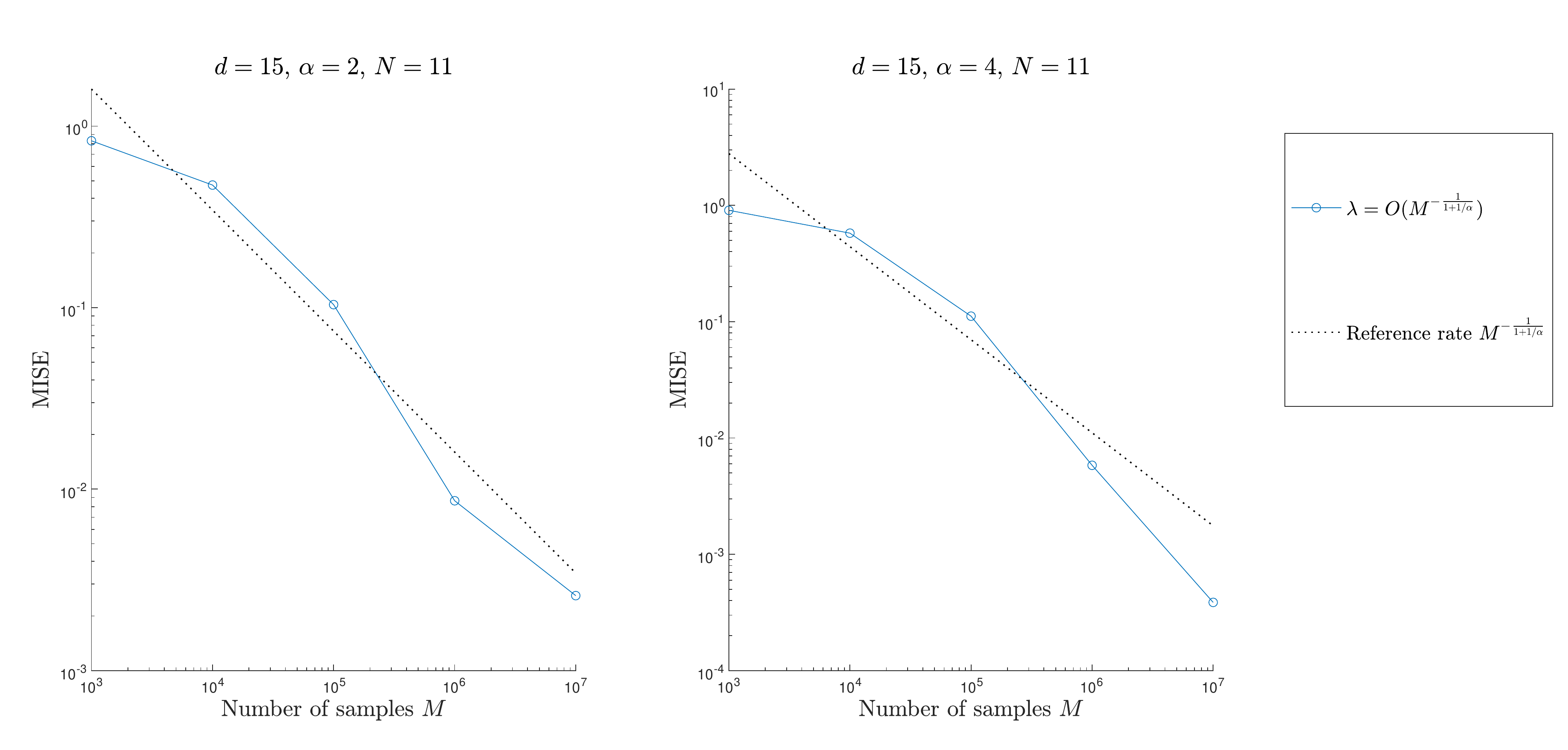}
		
		\caption{Plot of MISE for $d=15$ with $N=11$, $\alpha=2\text{ (left)},4\text{ (right)}$,
			and $\lambda=O(M^{-\frac{1}{1+1/\alpha}})$.}
		
		\label{fig:changelam(M)-d15}
	\end{figure}
	
	\section{Conclusions\label{sec:Conclusions}}
	
	In this paper, we considered a kernel method to approximate probability
	density functions. A major contribution of this paper is to have established
	a dimension-independent error convergence rate. To show this, we first
	developed a theory in a general RKHS setting. We then applied this
	theory to the Korobov-space setting, in which we established a dimension-independent
	error decay rate. 
	\rv{Therein, the implied constant is also dimension independent under suitable assumptions on the underlying weights.} 
	By choosing parameters suitably, we obtained the rate in terms of the MISE arbitrarily close to $1$ in the sample
	size, given that the target density is in the weighted Korobov space
	of arbitrary order. 
	\rv{For the Korobov spaces whose order is an even integer, the rate obtained is asymptotically minimax.} 
	Numerical results supported the theory.

	In closing, we discuss some possible future directions. An important point in the kernel density literature concerns the choice of the bandwidth (see for instance \cite{GoldenshlugerEtAl.A_2011_BandwidthSelectionKernel,GoldenshlugerEtAl.A_2014_AdaptiveMinimaxDensity} for robust adaptive selection, independent on the smoothness of the underlying density)
	In the present setting, our kernel is polynomial and not necessarily localised, so the problem of bandwidth selection does not really apply to our case. On the other hand, our results assume that the target density is in the RKHS corresponding to the kernel we use. Thus, a relevant question is how one should choose the weights in the kernel, as well as the regularization parameter $\lambda$, when the regularity class of the target density is not known. This topic is left for future work. 
	
	Another important point to discuss is the periodicity assumption. The full error estimates with respect to
	the sample size $M$ as in Corollary~\ref{cor:final} presented in
	this paper is limited to the weighted Korobov space setting, in which
	the target density is assumed to admit a smooth periodic extension.
	Needless to say, there is scope for further work to generalise our
	error bounds. In this regard, we stress that the results we established
	in Section~\ref{sec:gen-theory} are general and applicable to many
	other kernel functions. As we demonstrated in Section \ref{sec:Example},
	to obtain precise estimates as in Corollary~\ref{cor:final} for
	such kernels, it suffices to obtain a kernel interpolation estimate.
	Moreover, in view of the optimality of kernel interpolation, it suffices
	to obtain a sampling-based approximation that gives a small error
	in the corresponding reproducing kernel Hilbert space. In turn, this
	paper provides further motivations on fully discrete approximation
	methods.
	
	\section*{Acknowledgments}
	This research includes computations using the facilities of the Scientific IT and Application Support Center of EPFL.
	\printbibliography
\end{document}